\documentclass[11pt,leqno,a4paper]{amsart}
\usepackage{graphicx}
\usepackage{amsmath}
\usepackage{amsthm}
\usepackage{amssymb}
\usepackage{epsfig}
\usepackage[active]{srcltx}

\newcommand{\fineproof}{\hspace{\fill}$\square$}
\newcommand{\dsp}{\displaystyle}
\newcommand{\be}{\beta}
\newcommand{\pa}{\partial}
\newcommand{\s}{s}
\def \D {\Delta}
\def \d {\delta}

\def \al {\alpha}

\def \lm {\lambda}
\def \bd {\bigtriangledown}
\def \R {\Bbb R}
\def \ds {\displaystyle}
\def \O {\Omega}
\def \om {\Omega}
\def \p {\partial}

\def \ve {\varepsilon}
\def \o {\omega}
\newcommand{\scp}{\scriptscriptstyle}

\newcommand{\rdue}{\erre^2}
\def\erre{\mathbb{R}}
\newcommand{\ino}{\int_{\Omega}}

\newcommand{\beq}{\begin{equation}}
\newcommand{\eeq}{\end{equation}}
\newcommand{\graf}[1]{\left\{\begin{array}{ll}#1\end{array}\right.}

\newcommand{\ov}[1]{\overline{#1}}
\newcommand{\un}[1]{\underline{#1}}

\newtheorem{lemma}{Lemma}[section]
\newtheorem{Proposition}{Proposition}[section]
\newtheorem{theorem}{Theorem}[section]
\newtheorem{corollary}{Corollary}[section]
\newtheorem{remark}[lemma]{Remark}

\newtheorem{thm}{Theorem}

\date{}

\begin{document}
\title[Mean Field Equations on multiply connected domains]{Existence and uniqueness for Mean Field Equations on multiply connected domains
at the critical parameter}

\author{Daniele Bartolucci$^{(1,\ddag)}$ and Chang-Shou Lin$^{(2)}$}

\thanks{$^{(1)}$Daniele Bartolucci, Department of Mathematics, University of Rome {\it "Tor Vergata"},\\
Via della ricerca scientifica n.1, 00133 Roma, Italy. e-mail:bartoluc@mat.uniroma2.it}
\thanks{$^{(2)}$Chang Shou Lin,  Taida Institute for Mathematical Sciences and
Center for Advanced Study in Theoretical Sciences,
National Taiwan University, Taipei, Taiwan.\\ e-mail:cslin@math.ntu.edu.tw}

\thanks{$^{(\ddag)}$Research partially supported by FIRB project {\sl
Analysis and Beyond} and by MIUR project {\sl Metodi variazionali e PDE non lineari}}

\begin{abstract}

We consider the mean field equation:
\begin{equation*}
(1)\ \ \ \
    \left\{%
\begin{array}{ll}
    \Delta u+\rho\displaystyle\frac{e^u}{\int_\Omega e^u }=0 & \hbox{in} \;\Omega, \\
    u=0  & \hbox{on }\;\partial\Omega, \\
\end{array}%
\right.
\end{equation*}

\noindent where $\Omega\subset \mathbb{R}^2$ is an open and bounded
domain of class $C^1$. In his 1992 paper, Suzuki proved that if $\Omega$ is a
simply-connected domain, then equation (1) admits a unique solution for $\rho\in[0,8\pi)$.
This result for $\Omega$ a simply-connected domain
has been extended to the case $\rho=8\pi$ by
Chang, Chen and the second author.
However, the uniqueness result for $\Omega$ a multiply-connected domain has remained a
long standing open problem which we solve positively
here for $\rho\in[0,8\pi]$. To obtain this result we need a
new version of the classical Bol's inequality suitable to be applied
on multiply-connected domains.

Our second main concern is the existence of solutions for (1) when
$\rho=8\pi$. We a obtain necessary and sufficient condition for the
solvability of the mean field equation at $\rho=8\pi$ which is expressed in
terms of the Robin's function $\gamma$ for $\Omega$.
For example, if equation (1) has no solution at $\rho=8\pi$,
then $\gamma$ has a unique nondegenerate maximum point.\\
As a by product of our results we solve the long-standing open
problem of the equivalence of canonical and microcanonical
ensembles in the Onsager's statistical description of two-dimensional turbulence
on multiply-connected domains.
\end{abstract}

\maketitle

\section{Introduction}\label{sec1}
Let $\Omega\subset \R^2$ be an open and bounded domain of class $C^1$ and
$H^1_0(\Omega)$ denote the standard Sobolev space of functions with vanishing boundary values. We define
the functional $J_\rho:H^1_0(\Omega)\mapsto \R$ as
$$
J_\rho(u)=\frac{1}{2}\int_{\Omega}|\nabla u|^2-\rho\log\int_{\Omega}e^{u},\;\forall\;u\in
H^1_0(\Omega).
$$
The Euler-Lagrange equation for $J_\rho$ has the following form:
\beq\label{eqn1.1}
    \left\{%
\begin{array}{ll}
    \Delta u+\rho\displaystyle\frac{e^u}{\int_\Omega e^u }=0 & \hbox{in} \;\Omega, \\
    u=0  & \hbox{on }\;\partial\Omega, \\
\end{array}%
\right.
\eeq
where
$\Delta=\Sigma^2_{i=1}\frac{\partial^2}{\partial
x_i^2}$ is the Laplacian operation in $\mathbb{R}^2$. Problem \eqref{eqn1.1} is relevant
to many research areas in mathematics and it has been
extensively studied for the past three decades.
In geometry, the equation in \eqref{eqn1.1} is strictly related with the local version of the prescribed constant
Gaussian curvature problem on two dimensional surfaces see for example \cite{key01}, \cite{barjga},
\cite{key06}, \cite{key08}, \cite{key11}. In statistical mechanics, problem \eqref{eqn1.1} is the mean field
limit of the Gibbs measures associated with the Onsager's description of turbulent Euler flows, as studied by Caglioti,
Lions, Marchioro and Pulvirenti \cite{kof02}, \cite{kof03}, Kiessling \cite{K}, Chanillo and Kiessling \cite{key04}
and Lin \cite{key09}. Recently, it has attracted a lot of attention because it also appeared as a limiting equation
in the self-dual Chern-Simons-Higgs model in a relativistic version of superconductivity and other gauge field theories,
see \cite{bardem}, \cite{kof01}, \cite{kof04}, \cite{key03}, \cite{key13}, \cite{kof06}, \cite{kof07},
\cite{kof08}, \cite{kof09}, \cite{tar}, \cite{yang} and references therein.

The classical Moser-Trudinger \cite{moser} inequality implies that if $\rho\leq8\pi$,
then $J_\rho$ is bounded from below and coercive on $H^1_0(\Omega)$. In this situation
it is not hard to find a global minimizer of $J_\rho$ in $H_0^1(\Omega)$. Thus,
problem \eqref{eqn1.1} always admits at least one solution for $\rho<8\pi$. In general, we can compute a degree
counting formula for problem \eqref{eqn1.1} whenever $\rho\neq8\pi m$,
where $m$ is a positive integer. C.C. Chen and the second author
\cite{key05}, \cite{key06} proved that if $\Omega$ is not simply-connected, then the Leray-Schauder degree
corresponding to the resolvent operator naturally associated with \eqref{eqn1.1} does not vanish for any $\rho\neq8\pi m$.
Therefore \eqref{eqn1.1} admits at least one solution for $\rho\neq 8m\pi$ and $\Omega$ not simply-connected.
See also \cite{mald} for another derivation of the Leary-Schauder degree for \eqref{eqn1.1} on closed surfaces.

In case $\rho\neq8\pi m$, then the degree formulas obtained in \cite{key06} do depend only on
the topology of $\Omega$. On the contrary, if $\rho=8\pi m$, then the existence of solutions for \eqref{eqn1.1}
will depend also on the geometry of $\Omega$. For example if $\Omega$ is a ball,
then  \eqref{eqn1.1} has no solutions for
$\rho=  8\pi$, while if $\Omega$ is a long and thin ellipse and/or rectangle (see \cite{kof03} and in particular
\cite{key02}), then \eqref{eqn1.1} admits at least one solution for $\rho=8\pi$. Thus, it is natural
to ask the following question:

{\bf Q:} What kind of geometries do allow the existence of a solution for \eqref{eqn1.1} with $\rho=8\pi$?

In case $\Omega$ is simply-connected, Chang, Chen and the second author
\cite{key02} already gave an answer to this question. To
state their result, we should first recall the definition of the Robin's function for $\Omega$.
We let $G(x,\,p)$ denote the Green's function of $-\Delta$ with Dirichlet boundary
conditions, uniquely defined by
$$
\left\{
\begin{array}{lll}
-\Delta G(x,p)&=& \delta_{p}\quad \mbox{in}\quad \O, \\
\hspace{0.3cm}G(x,p)&=& 0 \quad \mbox{on}\quad \partial \O,
\end{array}\right.
$$
and set
\beq\label{eqn1.2}
    \left\{%
\begin{array}{ll}
    \widetilde{G}(x,\,p)=G(x,\,p)+\displaystyle\frac{1}{2\pi}\log|x-p|, \\
    \gamma(p)=\widetilde{G}(p,\,p). \\
\end{array}%
\right.
\eeq
Hence $\gamma$ denotes the Robin's function relative to $\Omega$ and satisfies
\[
\lim_{p\rightarrow\partial\Omega}\gamma(p)=-\infty.
\]
Let $q$ be a critical point of $\gamma(p)$. Clearly $q$ is also a
critical point of $\widetilde{G}(x,\,q)$ with respect to the $x$
variable, i.e.
\beq\label{eqn1.3}
  \nabla_{x} \widetilde{G}(x,\,q)=0\,\hbox{ at }\,x=q.
\eeq

Let $D(q)$ be defined by

\beq\label{eqn1.4}
    D(q)=\lim_{\varepsilon\rightarrow 0}\int_{\Omega\setminus\
    B(q,\,\varepsilon)}\frac{e^{8\pi(\widetilde{G}(x,q)-\gamma(q))}-1}{|x-q|^4}-\int_{\Omega^{\,c}}\frac{dx}{|x-q|^4},
\eeq

\noindent where $B(q,r)$ denotes the ball of center $q$ and radius
$r$. Note that in a neighborhood of $q$,
\beq\label{2606.1}
e^{8\pi(\widetilde{G}(x,q)-\gamma(q))}-1=\sum
a_{ij}(x_i-q_i)(x_j-q_j)+O(|x-q|^3),
\eeq
where, since $\widetilde{G}(x,q)$ is harmonic in $\Omega$, $a_{11} + a_{22}=0$. In particular, by using
\eqref{2606.1}, one can check that the limit in \eqref{eqn1.4} always exists.

Now we can state the main theorem in \cite{key02}.

\begin{thm}\label{ThmA}
Let $\Omega\subset\R^2$ be an open, bounded and simply-connected domain of class $C^{1}$.
Then \eqref{eqn1.1} admits at least one solution for $\rho=8\pi$ if and only if
there exists a maximum point $q$ of $\gamma$ such that $D(q)>0$.
\end{thm}

As an application of Theorem \ref{ThmA}, consider a dumbbell domain $\Omega_\delta$
with two disjoint balls $B_{1}$, $B_{2}$ connected by a tube of small width
  $\delta>0$.
Let $r_1$, $r_2$ be the radius of $B_1$ and $B_2$. If
  $r_2\neq r_1$, then, by using Theorem \ref{ThmA}, we can prove that \eqref{eqn1.1} with $\Omega\equiv\Omega_\delta$
has no solutions for $\rho=8\pi$ provided that $\delta$ is sufficiently small. However if
  $r_1=r_2$ and $\Omega_\delta$ is further assumed to be symmetric with respect to $y$-axis,
  then \eqref{eqn1.1} admits a solution for $\rho=8\pi$ for any $\delta$ sufficiently small.
  See \cite{key02} for a proof of these facts and further examples.

  One of our aims is to extend Theorem \ref{ThmA} to any bounded domain of class $C^1$. In fact we have

\begin{theorem}\label{Thm1.1}
Let $\Omega\subset \R^2$ be an open and bounded domain of class $C^{1}$. Then
\eqref{eqn1.1} admits at least one solution for $\rho=8\pi$ if and only if there exists a maximum point $q$ of
$\gamma$ such that $D(q)>0$.
\end{theorem}

An interesting application of Theorem \ref{Thm1.1} is the case
$\Omega=\Omega_\varepsilon=B(0,\,1)\setminus B(x_0,\,\varepsilon)$.
If $x_0=0$, then \eqref{eqn1.1} admits a solution (see for example \cite{kof03}) for $\rho=8\pi$ on
$\Omega_\varepsilon$ for any $\varepsilon\in (0,\,1)$.
However, if $x_0\neq 0$, then the Robin's function for $\Omega_\varepsilon$ (which we denote here
by $\gamma_\varepsilon (x)$) converges to the Robin's function for $B(0,\,1)$
(which we denote here by $\gamma(x)$) on any compact subset of $B(0,\,1)\setminus \{x_0\}$, see for example
\cite{BaFlu} p.198-199.
Since $0$ is the only maximum point of $\gamma$ then any maximum point $q_\varepsilon$
of $\gamma_\varepsilon$ must converge to $0$ as $\varepsilon\to 0$. But the quantity $D(0)$
relative to $B(0,\,1)$ is equal to $-1$ so that $D_\varepsilon (q_\varepsilon)<0$ provided that $\varepsilon$
is small. Hence Theorem \ref{Thm1.1} implies that \eqref{eqn1.1} on $\Omega=\Omega_\varepsilon$ has no
solutions at all for $\rho=8\pi$ and $\varepsilon$ small enough.

Consider the set of those bounded domains of class $C^1$ such that \eqref{eqn1.1} has no solutions for $\rho=8\pi$.
An interesting consequence of Theorem \ref{Thm1.1} is the closeness of this set of domains
under $C^1$ deformations. This property is essentially due to the fact that the quantity $D$ in \eqref{eqn1.3}
is stable under $C^1$ deformations.

\begin{corollary}\label{cor1.2}
Let $\{\Omega_n\}\subset \R^2$ be a sequence of open and bounded domains
of class $C^1$ and suppose that $\Omega_n$ converges to $\Omega$ in $C^1$ as
$n\rightarrow+\infty$, where $\Omega$ is an open and bounded
domain of class $C^1$ too. Suppose that equation \eqref{eqn1.1} on
$\Omega_n$ has no solutions for $\rho=8\pi$. Then equation
\eqref{eqn1.1} on $\Omega$ has no solutions for $\rho=8\pi$ as well.
\end{corollary}

Another interesting consequence of Theorem \ref{Thm1.1} is the deep
connection between the sign of $D(q)$ at a maximum point
$q$ of $\gamma$, the solvability of equation \eqref{eqn1.1} for $\rho=8\pi$
and the geometry of $\Omega$. As a consequence of \eqref{2606.1} we see that
$D(q)$ is well-defined whenever $q$ is a critical point of $\gamma$.
We will show in section \ref{sec4} that the following holds:

\begin{corollary}\label{cor1.3-add}
Let $\Omega\subset \R^2$ be an open and bounded domain of class $C^{1}$. If $D(p)\leq 0$ for a critical point $p$ of
$\gamma$, then $p$ must be a maximum point. In particular it is the unique
maximum point and is nondegenerate.
\end{corollary}

It is rather interesting to note that the Robin function is an elementary function of two variables
which carries some geometric information about $\om$ \cite{BaFlu},
but which has no apparent direct connections with \eqref{eqn1.1}. Nevertheless, we are not aware
of any "elementary" proof of the result in Corollary \ref{cor1.3-add} which make no use of \eqref{eqn1.1}.

\begin{remark}\label{rem200812}
As a straightforward consequence of Theorem \ref{Thm1.1} and Corollary \ref{cor1.3-add} we obtain a
result anticipated in the abstract, that is, if no solutions exist for \eqref{eqn1.1} with
$\rho=8\pi$ then the Robin function $\gamma$ for $\Omega$ admits a unique and nondegenerate maximum point.
\end{remark}

In other words, if
$\gamma$ has more than one maximum point, then the quantity
$D$, evaluated at any critical point of $\gamma$, must be
positive. Thus, we have another consequence of Theorem \ref{Thm1.1}.

\begin{corollary}\label{cor1.3}
Let $\Omega\subset \R^2$ be an open and bounded domain of class $C^{1}$.
Suppose that $\gamma$, the Robin's function for $\Omega$, has more than one maximum point.
Then  \eqref{eqn1.1} admits at least one solution at $\rho=8\pi$.
\end{corollary}

Finally, the following criterion turns out to be very useful \cite{key02}, \cite{BL2} to prove existence/non-existence
of a solution for $\rho=8\pi$. Let us define
$$
\mathcal{I}_{8\pi}(\O):=\inf\limits_{u\in H^{1}_0(\O)} J_\rho(u).
$$

\begin{corollary}\label{cor25.06} Let $\Omega\subset \R^2$ be an open, bounded and multiply-connected
domain of class $C^{1}$. Then
\beq\label{equalitynecsuf}
\frac{1}{8\pi}\mathcal{I}_{8\pi}(\O)\leq-1-\log(\pi)-4\pi\sup_{x\in\overline{\Omega}}\gamma(x),
\eeq
and \eqref{eqn1.1} admits at least one solution at $\rho=8\pi$ if and only if the strict inequality holds.
\end{corollary}

\begin{remark}
Corollary \ref{cor25.06} is false in general if we consider the analogue version of \eqref{eqn1.1}
on the flat two-torus with periodic boundary conditions. In fact, it has been shown in \cite{linwang},
\cite{linwang2} that if
the Green's function $G(\tau;0)$ for the torus has five critical points, then the equality holds in
\eqref{equalitynecsuf} and the analogue version of \eqref{eqn1.1} has one solution for $\rho=8\pi$,
which is a counterexample to Corollary 1.4 when the domain is a torus.
\end{remark}

We observe that both Theorem \ref{Thm1.1} and Corollary \ref{cor1.3} state the existence of at least one solution.
We can say much more concerning this point. Indeed, it turns out that the proof of Theorem \ref{Thm1.1}
heavily relies on the fact that \eqref{eqn1.1} admits at most one solution for $\rho\in [0,8\pi]$.
In case $\Omega$ is simply-connected, the
uniqueness of solutions for \eqref{eqn1.1} with $\rho\in[0,\,8\pi)$
has been proved by Suzuki in  \cite{key20}. That result has been improved
by Chang, Chen and the second author \cite{key02} for $\rho=8\pi$ and then
generalized by the authors in \cite{BL2} to cover the case where Dirac data are included
in \eqref{eqn1.1}. However, the uniqueness for \eqref{eqn1.1} on multiply-connected domains
has remained an open problem for a long time. In this paper, we
answer this question affirmatively.

\begin{theorem}\label{Thm1.2}
Let $\Omega$ be an open, bounded and multiply-connected domain of class $C^1$. Then
equation \eqref{eqn1.1} admits at most one solution for $\rho\leq8\pi$. Moreover,
the first eigenvalue of the corresponding linearized problem is strictly positive for any
$\rho\leq8\pi$.
\end{theorem}

The proof of Theorem \ref{Thm1.2} relies on our third new result which is the Bol's inequality on multiply-connected
domains. Let $\mathrm{w}\in C^2(\Omega)$ satisfy the differential inequality:
\beq\label{eqn1.6}
    \Delta \mathrm{w}+e^\mathrm{w}\geq 0 \,\hbox{ in }\,\Omega.
\eeq
For any relatively compact subdomain $\o\Subset \Omega$, we set
\beq\label{eqn1.7}
    m(\omega)=\int_\Omega e^\mathrm{w} dx,\quad \ell (\partial \omega)=\int_{\partial\omega}e^{\frac{\mathrm{w}}{2}}ds.
\eeq
The following inequality is the by now classical \cite{key01} Bol's isoperimetric inequality:
\begin{thm}\label{ThmB}
Suppose that $\Omega$ is an open, bounded and simply-connected domain and let
$\mathrm{w}\in C^2(\Omega)$ satisfy \eqref{eqn1.6}. Assume
\begin{equation}\label{eqn1.8}
    \int_\Omega e^\mathrm{w}dx\leq8\pi.
\end{equation}
Then $2\ell(\partial \omega)^2\geq m(\omega)(8\pi-m(\omega))$ for
any relatively compact subdomain $\omega\Subset\Omega$.
\end{thm}
Here and in the rest of this paper the notation $\omega\Subset\Omega$ will be always intended
to mean that $\omega$ is a relatively compact subdomain of $\Omega$.

For the sake of completeness we remark that Theorem A and the corresponding versions of
Corollaries \ref{cor1.2} and \ref{cor1.3} on simply connected domains has been generalized in \cite{BL2}
to cover the case where Dirac data are included in \eqref{eqn1.1}.
In particular, a version of Theorem B suitable to be applied to that singular case,
as well as the corresponding uniqueness result on simply connected domains has been obtained in \cite{BL1}
(see also \cite{key12}). For the corresponding existence and/or uniqueness questions on $\R^2$ or on the
flat two-torus we refer the reader to
\cite{key09}, \cite{key14}, \cite{linwang}.

We note that the assumption of simply-connectedness of $\Omega$ in
Theorem \ref{ThmB} is necessary. See the end of section \ref{sec2}
below for a counterexample to the Bol's inequality in case
$\Omega$ is an annulus. Actually our counterexample also shows
that the inequality may fail when solutions of \eqref{eqn1.6}
share some superhermonic part in the "hole" of the annulus.\\
Clearly, if $u$ solves \eqref{eqn1.1} then $\mathrm{w}=u-\log\int_\Omega e^u-\log \rho$ satisfies
\eqref{eqn1.6} with the equality sign.
This is way we extend Theorem \ref{ThmB} to the case where $\Omega$ is multiply-connected and
solutions of \eqref{eqn1.6} take constant values on $\partial \Omega$.

\begin{theorem}\label{Thm1.3}
Suppose that $\mathrm{w}\in C^2(\Omega)\cap C(\,\ov{\Omega}\,)$ satisfies \eqref{eqn1.6} with $\mathrm{w}=c$ on
$\partial\Omega$, for some constant value $c\in\R$. Then $2\ell(\partial \omega)^2\geq m(\omega)(8\pi-m(\omega))$
for any subdomain $\omega\Subset\Omega$. Furthermore, if $\omega$ is not simply-connected,
then the inequality is strict.
\end{theorem}

With the aid
of the Bol's inequality, then Theorem \ref{Thm1.2} is proved by following the arguments
due to Chang-Chen and the second author in \cite{key02}. However it seems that this procedure is not well-known and,
in particular, there is a subtle point which also requires some modification exactly
for the case $\rho=8\pi$. We will therefore provide a complete
proof of it in section \ref{sec3}.

A general remark is in order at this point.

\begin{remark}
Actually the proof of Theorem \ref{Thm1.1} requires an estimate for the sign of a first eigenvalue and
a uniqueness result suitable to be applied
to a larger class of equations than \eqref{eqn1.1}, see problem \eqref{eqn2.1} in section \ref{sec2}.
These are the content of Theorem \ref{thm3.1} in section \ref{sec3}, a truly more general result of
independent interest which calls up among other things for a new generalization of the Bol's inequality for
\eqref{eqn2.1} on multiply connected domains, see Theorem \ref{Thm2.1} below. The proof of Theorem \ref{Thm1.3}
will be derived as a straightforward consequence at the very end of section \ref{sec2}. A similar observation
holds for Theorem \ref{Thm1.1} and Corollaries \ref{cor1.3-add}, \ref{cor1.3} and \ref{cor25.06} which has not
been discussed in this introduction in their full generality to avoid technicalities. We refer to Theorem \ref{thm4.1}, Corollaries
\ref{cor4.1}, \ref{cor4.2} and \ref{cor25.06-gen} and Theorem \ref{thm4.2} for further details concerning this point.
\end{remark}

As a by product of our results, we are able to solve another long-standing open problem in the
rigorous statistical mechanics description of turbulent two-dimensional flows, see \cite{kof02}, \cite{kof03}
and \cite{K}. Two main variational tools has been used so far to understand the thermodynamical equilibrium
of two-dimensional turbulent Euler flows: the microcanonical and canonical variational principles, see
\eqref{mvp} and \eqref{cvp} in section \ref{sec5}. By using the uniqueness result in \cite{key20} and other results
already obtained in \cite{kof02}, in \cite{kof03} the authors were able to establish the \un{equivalence}
of these two variational principles, that is, the fact that they predict exactly the same thermodynamic.
This result was achieved under certain assumptions (see Proposition 3.3 in \cite{kof03}), one of which being the
simply-connectedness of $\om$. It seems that this restriction was entirely due to the fact that
uniqueness was known only on simply-connected domains. Therefore, as a corollary of Theorem \ref{Thm1.2}, we
are able to fill this gap and obtain the equivalence of microcanonical and canonical ensembles on multiply-connected
domains as well, see Theorem \ref{Equiv26.06} in section \ref{sec5} below. Actually, Theorem \ref{Thm1.1} and
Corollary \ref{cor25.06} provide an answer to another problem arising in \cite{kof03}, see Theorem \ref{Thm5.3}.
We refer to section \ref{sec5} for further details concerning this point.

\bigskip
\bigskip

This paper is organized as follows. In section \ref{sec2}, the Bol's inequality for multiply-connected
domains is proved together with the above mentioned counterexample. We will provide the proof
of Theorem \ref{Thm1.2} in section \ref{sec3}. A more general version of
Theorem \ref{Thm1.1} and Corollaries \ref{cor1.3-add}, \ref{cor1.3} and \ref{cor25.06} will be proved in section \ref{sec4}.
Finally section \ref{sec5} is devoted to the statistical mechanics applications.

\section{The Bol's inequality on multiply connected domains}\label{sec2}
\setcounter{equation}{0}

Let $\Omega$ be an open, bounded and multiply-connected domain of class $C^1$. In this section we
let $\overline{\Omega^\ast}$ be the closure of the union of the bounded components of
$\mathbb{R}^2\setminus\partial\Omega$ and
$\Omega^\ast=\overline{\Omega^\ast}\setminus\partial\overline{\Omega^\ast}$.
Clearly, $\Omega\subseteq\Omega^\ast$ and $\Omega^\ast\equiv\Omega$ if and only if $\Omega$ is simply-connected.\\
Actually the proof of Theorem \ref{Thm1.1} requires a uniqueness result suitable to be applied
to a larger class of equations than \eqref{eqn1.1}. Therefore we consider the more general problem
\begin{equation}\label{eqn2.1}
    \left\{%
\begin{array}{ll}
    \Delta u+\rho\displaystyle\frac{h(x)e^u}{\int_\Omega h(x)e^u}=0 & \hbox{ in }\,\Omega, \\
    u=0 & \hbox{ on }\,\partial\Omega, \\
\end{array}%
\right.
\end{equation}
where, here and in the rest of this paper, we assume that
\begin{eqnarray}\nonumber
&&h(x)
\hbox{\,\,is\,\,strictly positive\,\,and\,\,Lipschitz\,\,continuous\,\,in\,\,}\ov{\O^\ast}\\
&& \hbox{\,and\,\,}\log h(x)\,\,\hbox{is\,\,subharmonic\,\,in}\,\,\ov{\O^\ast}\label{eqn2.2}
\hbox{\,and\,\,harmonic\,\,in\,\,}\O.
\end{eqnarray}

Let $u$ be a solution of \eqref{eqn2.1}. We define
\begin{equation}\label{eqn2.3}
    \widehat{u}(x)=\left\{%
\begin{array}{ll}
    u(x) & \hbox{ if }\,x\in\Omega, \\
    0 & \hbox{ if }\,x\in\Omega^\ast\setminus\Omega. \\
\end{array}%
\right.
\end{equation}

\noindent Then $\widehat{u}$ satisfies the following inequality
\begin{equation}\label{eqn2.4}
    \Delta \widehat{u}+\rho\frac{h(x)e^{\widehat{u}}}{\int_\Omega
    h(x)e^{u}}\geq0\,\hbox{ in }\,\Omega^\ast
\end{equation}
\noindent in the sense of distributions. Although it is standard,
we would like to provide a proof of (\ref{eqn2.4}) here for the sake of completeness.

\begin{lemma}\label{lem2.1}
Let $\widehat{u}$ be defined by (\ref{eqn2.3}). Then $\widehat{u}$ satisfies
(\ref{eqn2.4}) in the sense of distributions.
\end{lemma}
\begin{proof}
Let $\varphi\in C^2_0(\overline{\Omega^\ast})$ and $\varphi\geq 0$
in $\Omega^\ast$. Then
\begin{eqnarray*}
    \int_{\Omega^\ast}(\Delta \varphi)\widehat{u}\,dx&=&\int_\Omega(\Delta\varphi) u \,dx\\
    &=&\int_\Omega\varphi(\Delta u) dx-\int_{\partial\Omega}\varphi(x)\frac{\partial u}{\partial
    \nu}(x)d\sigma,
\end{eqnarray*}
\noindent where $\nu(x)$ is the outer unit normal on $\partial\Omega$ at
$x\in \partial\Omega$. Since $u(x)>0$ in $\Omega$, then $\frac{\partial
u}{\partial \nu}(x)<0$ on $\partial\Omega$ by the strong maximum principle. Hence
\begin{eqnarray*}
    &&\int_{\Omega^\ast}(\Delta\varphi) \widehat{u}\, dx+\frac{\rho}{\int_{\Omega}he^u}h e^{\widehat{u}}\varphi dx\\
    &&\geq\int_\Omega\varphi\left(\Delta u+\rho\frac{h e^u}{\int_\Omega he^u}\right)dx
    -\int_{\partial\Omega}\varphi(x)\frac{\partial u}{\partial
    \nu}(x)d\sigma\geq0.
\end{eqnarray*}
\end{proof}

Putting $\beta=\log \rho -\log\int_\Omega h(x)e^{u(x)}$ and
\begin{eqnarray}\label{eqn2.5}
    v(x)&=&u(x)+\log h(x)+\log \rho -\log\int_\Omega h(x)e^{u(x)}\nonumber\\
        &=&u(x)+\log h(x)+\beta,\,x\in\Omega,
\end{eqnarray}
we see that, as a consequence of \eqref{eqn2.1} and \eqref{eqn2.2}, $v(x)$ satisfies
\beq\label{eqn2.6}
    \Delta v+e^{v}= 0\,\hbox{ in }\,\Omega,
\eeq
in the sense of distributions.
Let $\omega$ be a subdomain of $\Omega$ and set
$$
    m(\omega)=\int_\omega e^v
    dx,\,\,\,\,\,\ell(\partial\omega)=\int_{\partial\omega}e^{\frac{v}{2}} ds.
$$

\begin{theorem}\label{Thm2.1}
Suppose that $\Omega$ is an open, bounded and multiply-connected domain of class $C^1$ and $u$
is a solution of \eqref{eqn2.1} with $\rho\leq8\pi$. Then
$2\ell(\partial \omega)^2\geq m(\omega)(8\pi-m(\omega))$ for any
subdomain $\omega\Subset\Omega$. Furthermore, if $\omega$ is not simply-connected,
then the inequality is strict.
\end{theorem}

We extend $v(x)$ on $\Omega^\ast$ by defining
$$
\widehat{v}(x)=\widehat{u}(x)+\log{h(x)} +\beta,\,x\in \Omega^\ast.
$$
As a consequence of Lemma \ref{lem2.1} and \eqref{eqn2.2}, $\widehat{v}$ satisfies
\beq\label{0612.1}
    \Delta \widehat{v}+e^{\widehat{v}}\geq0\,\hbox{ in }\Omega^\ast
\eeq
in the sense of distributions. Clearly Theorem \ref{ThmB} could be applied
on any simply-connected subdomain $\omega_0\Subset\Omega^\ast$ such that
$\int_{\omega_0}e^{\widehat{v}}dx\leq 8\pi$, whenever $\widehat{v}\in C^2(\Omega^\ast)$.
In our case however $\widehat{v}$ is just Lipschitz continuous in
$\Omega^\ast$, which is why we need the following:

\begin{lemma}\label{lem2.2}
Put
\begin{eqnarray*}
    \widehat{m}(\omega)=\int_\omega e^{\widehat{v}}
    dx,\,\,\,\,\, \widehat{\ell}(\partial\omega)=\int_{\partial\omega}e^{\frac{\widehat{v}}{2}} ds,
\end{eqnarray*}
and let $\omega\subseteq \omega_0\Subset \Omega^*$, where $\omega_0$ is
a simply-connected domain. If $\widehat{m}(\omega_0)\leq 8\pi$, then
\begin{equation}{\label{eqn2.7}}
2\widehat{\ell}(\p\o)^2 \geq \widehat{m}(\o) (8\pi-\widehat{m}(\o)).
\end{equation}
\end{lemma}

\begin{proof}
Without loss of generality, we may assume $\widehat{m}(\o_0)<8\pi$. Let
$\psi_\ve(x)$ be a suitable mollifier, that is,
$C^{\infty}_0(\R^2)\ni\psi_\ve(x)\geq 0$, $\psi_\ve(x)=0$ for $|x|\geq \ve$ and  $\int\limits_{\R^2}\psi_\ve(x)=1$.
Set $v_\ve(x) = (\psi_\ve * \widehat{v})(x)$ to be the standard convolution. Thus, by using \eqref{0612.1}
we obtain
$$
\D v_\ve(x) + (\psi_\ve * e^{\widehat{v}})(x) \geq 0\ \ \mbox{ in }\ \ \O^\ast,
$$
in the sense of distributions. Since $v_\ve(x)$ uniformly converges to $\widehat{v}(x)$ as $\ve\to 0^+$, then there exists a
small constant $\d(\ve)>0$ such that
$$
(1+\delta(\ve)) e^{v_\ve(x)} \geq \psi_\ve(x) * e^{\widehat{v}(x)}\ \ \mbox{for}\ \ x\in \ov{\o_0}.
$$
Therefore $v_\ve(x)$ also satisfies
$$
\D v_\ve(x) + (1+\delta(\ve)) e^{v_\ve(x)} \geq 0 \ \ \mbox{in}\ \ \ov{\o_0}.
$$
For any $\ve$ small enough we have
$$
\int_{\ov{\o_0}} (1+\delta(\ve)) e^{v_\ve(x)} dx <8\pi.
$$
Let $m_\ve(\o)$ and $\ell_\ve(\p\o)$ be defined in the usual way in terms of the
metric $(1+\delta(\ve)) e^{v_\ve(x)}$ and $\sqrt{1+\delta(\ve)}
e^{\frac{v_\ve(x)}{2}}$ respectively. Then Theorem B implies
$$
2\ell_\ve^{\,2} (\p\o) \geq m_\ve (\o) (8\pi -m_\ve(\o)),
$$
and we obtain \eqref{eqn2.7} by passing to the limit as $\ve\to 0^+$.
\end{proof}

\textbf{The Proof of Theorem 2.1.}\\
Once the Bol's inequality has been established,
then the fact that it is strict for domains which are not simply-connected can be proved
by arguing exactly as in \cite{key02}. Since there is nothing new concerning this point
we refer the reader to that paper for further details.\\
Next, observe that if $\o$ is simply-connected then the conclusion easily follows from Theorem B.
Therefore we assume
without loss of generality that $\o$ is multiply-connected and
first consider the case where $\p\o$ does not bound neither one simply-connected subdomain of $\O$.
In this situation, each bounded component of $\R^2\backslash \p \o$
contains at least one bounded component of $\R^2\backslash \p\O$.\\
Let $\O_0$ be the
union of those bounded components of $\R^2\backslash \p\O$ which
are bounded by $\p\o$. Let $\o_0$ be the union of all bounded
simply-connected components of $\R^2\backslash \p \o$. Thus
$\O_0\subset \o_0$ and we define
$$
\o^*=\o_0\backslash \O_0\cup \o.
$$
Clearly $\o^*\subset \O$. See Figure \ref{fig03}.

\begin{figure}[hbt]
\centering \graphicspath{{pipic/}}
\includegraphics[trim=50 0 80 0,clip,width=3in,height=2.7in]{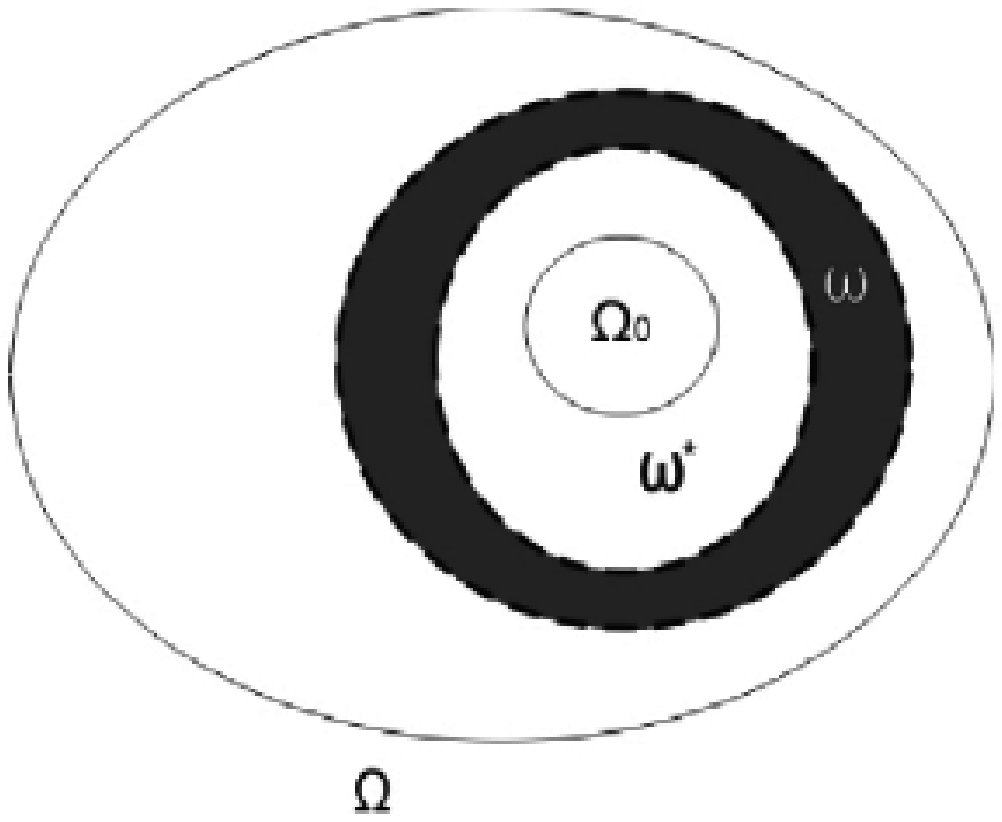}
\caption{}\label{fig03}
\end{figure}



In particular we see that both $\o^* \cup \O_0$ and
$\o^\ast\cup\o\cup\O_0$ are simply-connected domains. Let $\p_0\o$ be the boundary of $\o^* \cup \O_0$
and $\p_1\o=\p\o\setminus\p\o^*$. Then $\p_1\o=\p(\o^*\cup\o\cup\O_0)$ and we have
$$
\p \o = \p_1 \o \cup \p_0 \o.
$$

Next we discuss three cases separately.

\noindent {\bf Case 1.}\ \ \ $\widehat{m}(\o^* \cup \O_0) \geq 8\pi.$\\

Since $u(x)>0$ in $\O$, we have

\begin{eqnarray*}
2\ell(\p_0\o)^2 & = & 2\Biggl( \int_{\p_0\o} e^{\frac {u(x)}{2}+\frac{1}{2}\log h(x)} e^{\frac \beta 2} ds \Biggl)^2\\
& \geq & 2 \Biggl( e^{\frac \beta 2} \int_{\p_0\o}
e^{\frac{1}{2}\log h(x)}ds\Biggl)^2.
\end{eqnarray*}

\noindent
Since $\log h(x)$ is subharmonic in $\O^*$ and $\p_0\o=\p(\o^*\cup\O_0)$, then
\[
\log h(x)\leq g(x),\,x\in \o^*\cup\O_0,
\]
where $g(x)$ is the harmonic (in $\omega^*\cup\O_0$) function which also satisfies
\[
g(x)=\log h(x),\,x\in\p_0\o.
\]
Since $\omega^*\cup\O_0$ is simply-connected, by the Nehari's inequality \cite{Ne}, we have
\[
\left(\int_{\p_0\o}e^{\frac{1}{2}\log h(x)}ds\right)^2=
\left(\int_{\p_0\o}e^{\frac{g(x)}{2}}ds\right)^2\geq4\pi\int_{\o^*\cup\O_0}e^{g(x)}dx.
\]

Hence,
\begin{eqnarray}\label{eqn2.8}
    2\ell(\p_0\o)^2&\geq&8\pi\int_{\o^*\cup\O_0}e^{g(x)+\beta}dx
    \geq8\pi\int_{\o^*\cup\O_0}e^{\log h(x)+\beta} dx\\
    &>&8\pi\int_{\O_0}h(x)e^{\beta}dx=8\pi \widehat{m}(\O_0)\nonumber
\end{eqnarray}

Since $\widehat{m}(\o^*\cup\O_0)\geq 8\pi$, we have
\[
\widehat{m}(\O_0)\geq8\pi-\widehat{m}(\o^*)\equiv8\pi-m(\o^*),
\]
and then \eqref{eqn2.8} yields
\begin{equation}\label{eqn2.9}
2\ell^2(\p_0\o)\geq8\pi(8\pi-m(\o^*)).
\end{equation}

Since $m(\o\cup  \o^*\cup \O_0) > m(\o^* \cup \O_0)\geq 8\pi$,
the same argument with minor modifications can be used to obtain
\beq{\label{eqn2.10}}
2\ell^2(\p_1\o) \geq 8\pi (8\pi-m(\o\cup \o^*)).
\eeq
Hence, by using \eqref{eqn2.9} and (\ref{eqn2.10}), we conclude that
\begin{eqnarray*}
2\ell^2(\p\o) & = & 2(l(\p_1\o)+l(\p_0\o))^2 > 2 [l^2(\p_1\o)+l^2(\p_0\o)]\\
& \geq & m(\o^*) (8\pi-m(\o^*)) + m(\o\cup\o^*) (8\pi -m(\o\cup\o^*))\\
& = & (m(\o)+m(\o^*)) (8\pi - m(\o) -m(\o^*)) + m(\o^*) (8\pi-m(\o^*))\\
& = & m(\o) (8\pi-m(\o)) + m(\o^*) (16\pi -2m(\o) -2m(\o^*)).
\end{eqnarray*}

Since
$$
m(\o)+m(\o^*)\leq m(\O) \leq 8\pi,
$$
then
$$
2\ell^2(\p\o)
> m(\o) (8\pi-m(\o)),
$$ which proves Theorem \ref{Thm2.1} in Case 1.

\noindent {\bf Case 2.}\ \ \ $\widehat{m}(\o^*\cup \O_0)<8\pi$ and $\widehat{m}(\o\cup\o^*\cup \O_0)\geq 8\pi$.\\

Since $\widehat{m}(\o^*\cup \O_0)<8\pi$ and $\o^*\cup \O_0$ is simply-connected,
we can apply Lemma \ref{lem2.2} to obtain
\begin{eqnarray}{\label{eqn2.11}}
2\ell^2(\p_0\o) & \geq & \widehat{m}(\o^*\cup \O_0)(8\pi-\widehat{m}(\o^*\cup \O_0))\\
 & = & (m(\o^*)+\widehat{m}(\O_0)) (8\pi-m(\o^*)-\widehat{m}(\O_0)).\nonumber
\end{eqnarray}
As a consequence of (\ref{eqn2.8}) (which is easily seen to be satisfied in Case 2 as well) we conclude that
$\ell(\p_0\o)$ and $\ell(\p_1\o)$ satisfy
\beq{\label{eqn2.12}}
\ell(\p_0\o) \geq \sqrt{4\pi \widehat{m}(\O_0)} \ \ \mbox{and}\ \ \ell(\p_1\o) \geq \sqrt{4\pi \widehat{m}(\O_0)}.
\eeq

Clearly (\ref{eqn2.10}) holds in Case 2 as well and therefore it can
be used together with (\ref{eqn2.11}) and
(\ref{eqn2.12}) to conclude that
\begin{eqnarray*}
2\ell^2(\p\o) & = & 2[\ell^2 (\p_1\o)+2\ell(\p_1\o)\ell(\p_0\o)+\ell^2(\p_0\o)]\\
& \geq & 8\pi (8\pi -m(\o)-m(\o^*)) \\
& & + (m(\o^*)+\widehat{m}(\O_0)) (8\pi-m(\o^*)-\widehat{m}(\O_0))+16\pi \widehat{m}(\O_0)\\
& = & 8\pi (8\pi -m(\o)) - m^2(\o^*) + \widehat{m}(\O_0) (24\pi -2m(\o^*)-\widehat{m}(\O_0))\\
& = & m(\o) (8\pi-m(\o))+[(8\pi-m(\o))^2 - m^2(\o^*)]\\
& &  + \widehat{m} (\O_0) (24\pi-2m(\o^*)-\widehat{m}(\O_0)).
\end{eqnarray*}

\noindent Since $m(\o)+m(\o^*)\leq m(\O)\leq8\pi$, then
\[
m(\o^*)\leq8\pi-m(\o).
\]
Moreover since $\widehat{m}(\o^*\cup\O_0)<8\pi$, then
\[
2m(\o^*)+2\widehat{m}(\O_0)<16\pi.
\]

Hence $2\ell^2(\p\o)>m(\o)(8\pi-m(\o))$ which proves Theorem
\ref{Thm2.1} in Case 2 as well.

\noindent {\bf Case 3.}\ \ \ $\widehat{m}(\o\cup \o^* \cup \O_0)<8\pi$.

Since $\o\cup\o^*\cup\O_0$ is simply-connected, by applying Lemma
\ref{lem2.2} to $\o$, we have
\[
2\ell(\p\o)^2\geq m(\o)(8\pi-m(\o)),
\]
which concludes the proof of Theorem \ref{Thm2.1} in case
the interior of $\p\o$ does not contain neither one simply-connected subdomain of $\O$.

Now suppose that $\p\o$ bounds some simply-connected subdomains of $\O$ which
we denote by $\o_1,\ldots,\o_k$ with $k\geq 1$. Then $\o\cup\ov{\o_1}\cup\dots\ov{\o_k}$ is connected in
$\O$ and its boundary does not bound any simply-connected
component in $\O$. In particular
$$
\p\o = \p (\o\cup \o_1\cup\ldots\cup \o_k)\setminus \bigcup_{j=1}^k (\p\o_j).
$$
By applying the previous result to $\o\cup \o_1 \cup \ldots \cup \o_k$ and Lemma \ref{lem2.2} (or either
Theorem B) to the domains  $\o_j, 1\leq j\leq k$, we obtain
\begin{eqnarray}{\label{eq:2.13}}
2\ell^2(\p(\o\cup\o_1,\cup\ldots\cup \o_k)) & \geq & m(\o\cup
\o_1\cup \ldots \cup \o_k)\\ \nonumber & & [8\pi - m
(\o\cup\o_1,\ldots,\cup \o_k)],
\end{eqnarray}
\begin{equation}{\label{eq:2.14}}
2\ell^2 (\p \o_j) \geq m(\o_j) (8\pi - m(\o_j)).
\end{equation}

If $k=1$, by using \eqref{eq:2.13} and \eqref{eq:2.14} we have
\begin{eqnarray*}
2\ell(\p\o)^2 & > & 2\ell (\p(\o\cup\o_1))^2+2\ell(\p\o_1)^2\\
& \geq & m(\o\cup\o_1) (8\pi -m(\o\cup\o_1))+m(\o_1)(8\pi-m(\o_1))\\
& = & m(\o) (8\pi-m(\o_1))+m(\o_1) (16\pi-2m(\o_1)-2m(\o))\\
& \geq & m(\o) (8\pi-m(\o)).
\end{eqnarray*}
We omit the details of the case $k>1$ which is worked out by similar arguments.\hspace{\fill}$\square$

\bigskip
\bigskip

At this point we are ready to provide the following:

\noindent{\bf Proof of Theorem \ref{Thm1.3}}\\
Since $\mathrm{w}\in C^2(\Omega)\cap C(\,\ov{\Omega}\,)$ satisfies \eqref{eqn1.6} and
$\mathrm{w}=c$ on $\partial \Omega$ then, by arguing exactly as in Lemma \ref{lem2.1},
it is easy to check that setting $u=\mathrm{w}-c$ and
$\widehat{u}$ to be the corresponding (vanishing) extension to $\O^*$ as defined in \eqref{eqn2.3},
then it satisfies
$$
\Delta \widehat{u}+e^{c}e^{\widehat{u}}\geq 0 \,\hbox{ in }\,\O^*,
$$
in the sense of distributions. Hence $\widehat{v}=\widehat{u}+c$ is Lipschitz continuous in $\O^*$ and
satisfies \eqref{0612.1} in the sense of distributions.
The rest of the proof of Theorem \ref{Thm2.1} (including of course that of Lemma \ref{lem2.2})
works without any further modification and the desired conclusion follows.\hspace{\fill}$\square$

\bigskip

We conclude this section with an example which shows at the same time that if $\O$ is not simply-connected then
Theorem B fails and if $\log h$ cannot be extended to a subharmonic function in $\O^*$ (see \eqref{eqn2.2})
then Theorem \ref{Thm2.1} fails.\\

\noindent{\bf Example} [{\it Failure of the Bol's inequality on multiply-connected domains}]\\
For any $-1<\al<0$ and $a>0$ let us define
$$
v_\al(x)=\log{\left(\frac{8(1+\al)^2a^2|x|^{2\al}}{(1+a^2|x|^{2(1+\al)})^2}\right)},\;\;x\in\R^2.
$$
Observe that $v_\al$ satisfies $-\Delta v_a= e^{v_\al} -4\pi\al\delta_{p=0}$ in the sense of distributions in
$\R^2$ and in particular $-\Delta v_a= e^{v_\al}$ classically in $\R^2\setminus\{0\}$.
Since $v_\al\simeq 2\al\log|x|$ as $x\to 0$ it is clear that its leading term is superharmonic near the origin.
For each $0<s<t<+\infty$ let us set
$$
A_{s,t}=\{x\in\R^2\,|\,s<|x|<t\},
$$
and for $0<R_1<r_1<r_2<R_2<+\infty$ let us define
$$
\Omega:=A_{R_1,R_2},\qquad \omega=A_{r_1,r_2}.
$$
We have
$$
\ell(\omega)=\int\limits_{\{|x|=r_1\}\cup \{|x|=r_2\}}
\frac{\sqrt{8}(1+\al) a |x|^{\al}}{1+a^2|x|^{2(1+\al)}}\,ds=
$$
$$
\frac{2\pi\sqrt{8}(1+\al) a r_1^{\al+1}}{1+a^2r_1^{2(1+\al)}}+
\frac{2\pi\sqrt{8}(1+\al) a r_2^{\al+1}}{1+a^2r_2^{2(1+\al)}},
$$
and
$$
m(\omega)=\int\limits_{A_{r_1,r_2}}\frac{8(1+\al)^2a^2|x|^{2\al}}{(1+a^2|x|^{2(1+\al)})^2}\,dx=
$$
$$
8\pi(1+\al)\left(\frac{1}{1+a^2r_1^{2(1+\al)}}-\frac{1}{1+a^2r_2^{2(1+\al)}}\right).
$$

Clearly the last identity implies
$$
\int\limits_{\Omega} e^{v_\al}<\int\limits_{\R^2} e^{v_\al}=8\pi(1+\al)<8\pi.
$$

Therefore, letting $0<R_1<r_1\searrow 0$, we obtain
$$
2 \ell^2(\omega)= 64\pi^2(1+\al)^2 \frac{a^2r_2^{2(1+\al)}}{\left(1+a^2r_2^{2(1+\al)}\right)^2}+o(1),
$$
and
$$
m(\omega)(8\pi(1+\al)-m(\omega))=
$$
$$
64\pi^2(1+\al)^2
\left(\frac{a^2r_2^{2(1+\al)}}{\left(1+a^2r_2^{2(1+\al)}\right)}+o(1)\right)
\left(\frac{1}{\left(1+a^2r_2^{2(1+\al)}\right)}+o(1)\right).
$$

We readily conclude that as $0<R_1<r_1\searrow 0$ it holds
$$
2 \ell^2(\omega)=m(\omega)(8\pi(1+\al)-m(\omega))+o(1),
$$
and then, for any fixed $-1<\al<0$ and any $0<R_1<r_1$ small enough we see that
the inequality
$$
2 \ell^2(\omega)=m(\omega)(8\pi-m(\omega))+8\pi\al m(\omega)+o(1)\leq
$$
$$
m(\omega)(8\pi-m(\omega))+4\pi\al m(\omega)<m(\omega)(8\pi-m(\omega)),
$$
holds. As a consequence the Bol's inequality does not hold in the situation at hand.

\section{Symmetrization and uniqueness}\label{sec3}
\setcounter{equation}{0}

The main theorem in this section is the following.

\begin{theorem}\label{thm3.1}
Suppose that $\Omega\subset \R^2$ is a bounded domain of class $C^1$
and $h(x)$ satisfies \eqref{eqn2.2}. Then, for any $\rho\leq8\pi$,
there exists at most one solution for problem \eqref{eqn2.1}. In particular the first eigenvalue
of the linearized problem for \eqref{eqn2.1} is strictly positive for any $\rho\leq8\pi$.
\end{theorem}

As a matter of fact, the proof of Theorem \ref{thm3.1} can be worked out by using a rearrangement argument
together with the improved (see Theorem \ref{Thm2.1}) Bol's inequality.
As mentioned above, the first proof for $\rho<8\pi$ on simply-connected domains
was given by Suzuki in \cite{key20}. That argument
was improved in \cite{key02} to cover the case $\rho\leq 8\pi$ and later in \cite{BL1} where the more general
situation of singular (Dirac) data was considered as well. However that
rearrangement argument seems not to be well-known and in particular the argument used to handle
the more subtle case, that is $\rho=8\pi$, has to be modified with respect to the one
adopted in \cite{key02}, \cite{BL1}. Hence we will present a complete proof here
for the sake of completeness.

\begin{proof}
Since Theorem \ref{thm3.1} is well
known for $\O$ simply-connected, with the unique exception of the incoming
statement of Lemma \ref{lem3.1} we will assume that $\O$ is not
simply-connected for the rest of this section. In any case $\O$ will be always assumed to be \un{open,
bounded and of class $C^1$}.\\
For any such $\O$, and for fixed $V\in L^{\infty}(\O)$, we will say that $\lm_k=\lm_k(V,\O)$ is the
$k$-th eigenvalue of $\triangle+V$ if there exists $\psi_k\in H^{1}_0(\O)$ such that
$$
-\triangle \psi_k -V \psi_k=\lm_k V \psi_k \quad\mbox{in}\quad \O.
$$
We begin with the following Lemma of independent interest. The first part is well known see
\cite{key01}, \cite{key20} and more recently \cite{BL1}, while the second is the
novel generalization of those results to the case where $\O$ is multiply-connected and of class $C^1$.

\begin{lemma}\label{lem3.1}$\,$\\
$({\bf \mathrm{I}})$ Let $\mathrm{w}$ satisfy \eqref{eqn1.6} and $\O$ be simply-connected.
Then $\lm_1(e^\mathrm{w},\Omega)>0$ whenever $\int\limits_{\O}e^\mathrm{w}\leq 4\pi$ while
$\lm_2(e^\mathrm{w},\Omega)>0$ whenever $\int\limits_{\O}e^\mathrm{w}\leq 8\pi$.\\
\noindent
$({\bf \mathrm{II}})$ Let $v$ take the form \eqref{eqn2.5} and therefore satisfy \eqref{eqn2.6} on
the multiply-connected domain $\O$.\\
Then $\lm_1(e^v,\Omega)>0$ whenever $\rho\leq 4\pi$ while
$\lm_2(e^v,\Omega)>0$ whenever $\rho\leq 8\pi$.
\end{lemma}

\begin{proof}
As mentioned above $({\bf \mathrm{I}})$ is well known, see for example \cite{BL1} for a detailed proof.\\
We will first prove the assertion $({\bf \mathrm{II}})$ concerning $\lm_2\equiv\lm_2(e^v,\Omega)$.
We argue by contradiction and suppose that $\lm_2\leq0$. Then there exists
$K\leq 1$ and a second eigenfunction $\varphi$ satisfying
\beq\label{eqn3.1}
\left\{%
\begin{array}{ll}
    \Delta\varphi+Ke^v\varphi=0 &\, \hbox{ in } \;\Omega, \\
    \varphi=0, &\, \hbox{ on } \; \partial\Omega. \\
\end{array}%
\right.
\eeq

Let $\Omega^+=\{x\in\Omega|\varphi(x)>0\}$. We want to prove that

\begin{equation}\label{eqn3.2}
\int_{\Omega_{+}}e^{v(x)}dx\geq4\pi.
\end{equation}

 Set $U(x)=-2\log(1+\frac{1}{8}|x|^2)$, which is an entire solution of
\beq\label{eqn3.3}
  \Delta U + e^{U}=0 \ \hbox{ in } \;\mathbb{R}^2.
\eeq

\noindent For any $t>0$, set
$\Omega_t=\{x\in\Omega|\varphi(x)>t\}\Subset\Omega$, and let $r(t)$
be uniquely defined by the equality
\begin{equation}\label{eqn3.4}
\int_{B_{r(t)}}e^{U(x)}dx=\int_{\{\varphi>t\}}e^{v(x)}dx,
\end{equation}

\noindent where $B_{r(t)}$ is the open ball of center $O$ and
radius $r(t)$. Since $\varphi$ is smooth, $r(t)$ is a strictly
decreasing and continuous function of
$t\in[0,\max_{\Omega}\varphi]$. Let
\begin{equation}
\varphi^\ast(r) = \sup_{t>0}\lbrace t \vert  v <r(t)\rbrace.
\end{equation}
Thus, for $t \in [0,\max_{\begin{tiny}\Omega\end{tiny}}\varphi]$, the equalities
\begin{equation}
\varphi^\ast(r(t))\;=\;t \quad \mbox{and} \quad r(\varphi^\ast(r)) \;=\;
r
\end{equation}
hold. Clearly \eqref{eqn3.4} implies
\begin{equation}\label{eqn3.5}
\int_{\{\varphi^\ast>t\}} e^{U(x)}dx = \int_{\{\varphi>t\}} e^{v(x)}dx,
\end{equation}
and
\begin{equation}\label{eqn3.6}
\int_{B_{R_{0}}} e^{U(x)}(\varphi^\ast)^2 dx= \int_{\Omega_{+}}
e^{v(x)}\varphi^2dx ,
\end{equation}
where $R_0=r(0)$. To derive a contradiction, we use the coarea
formulas
\begin{eqnarray}\label{eqn3.7}
&&-\frac{d}{dt}\int_{\Omega_{t}} |\nabla\varphi |^{2} dx = \int_{\partial\Omega_t}|\nabla\varphi| ds, \ \mbox{and}\\
&&-\frac{d}{dt}\int_{\Omega_{t}} e^{v(x)}dx =
\int_{\partial\Omega_t} \frac{e^v}{|\nabla \varphi|}ds \nonumber
\end{eqnarray}
which hold simultaneously for almost any $t$. Since $\int_{\Omega} e^v dx=\rho\leq8\pi$,
then \eqref{eqn3.7}, the Cauchy-Schwarz inequality and Theorem \ref{Thm2.1} together imply
\begin{eqnarray}\label{eqn3.8}
&&\hspace{1cm}-\frac{d}{dt}\int_{\Omega_{t}} |\nabla\varphi|^2 dx = \int_{\{\varphi=t\}}|\nabla\varphi| ds \ \\
\nonumber
&& \geq  \left(\int_{\{\varphi=t\}} e^{v/2}ds\right)^2 \left(\int_{\{\varphi=t\}} \frac{e^v}{|\nabla\varphi|}ds\right)^{-1}\\
\nonumber
&&=\ell^2\left(\lbrace\varphi=t\rbrace\right)\left(-\frac{d}{dt}\int_{\Omega_{t}} e^{v(x)}dx\right)^{-1}\\
\nonumber
&&\geq \frac{1}{2}\left(8\pi-\int_{\Omega_{t}}e^{v(x)}dx\right)\left(\int_{\Omega_{t}}e^{v(x)}dx\right)
\left(- \frac{d}{dt}\int_{\Omega_{t}}e^{v(x)}dx\right)^{-1}\\
\nonumber
&&=\frac{1}{2}\left(8\pi-\int_{\{\varphi^\ast>t\}}e^{U(x)}dx\right)
\left(\int_{\{\varphi^\ast>t\}}e^{U(x)}dx\right)
\left(- \frac{d}{dt}\int_{\{\varphi^\ast>t\}}e^{U(x)}dx\right)^{-1}.
\end{eqnarray}

The same computation for $\bigtriangledown\varphi^\ast$ yields,

\begin{eqnarray}\label{eqn3.9}
&&\hspace{1cm}-\dfrac{d}{dt}\int_{\{\varphi^\ast>t\}}\vert\nabla\varphi^\ast\vert^2
dx \\&&=
\frac{1}{2}\left(8\pi-\int_{\{\varphi^\ast>t\}}e^{U(x)}dx\right)\left(\int_{\{\varphi^\ast>t\}}e^{U(x)}dx\right)
\left(- \dfrac{d}{dt}\int_{\{\varphi^\ast>t\}}e^{U(x)}dx\right)^{-1},\nonumber
\end{eqnarray}
for the same values of $t$, except possibly for a set of null measure.
Therefore,
\beq
-\dfrac{d}{dt}\int_{\{\varphi>t\}}\vert\nabla\varphi\vert^2 dx \geq
-\dfrac{d}{dt}\int_{\{\varphi^\ast>t\}}\vert\nabla\varphi^\ast\vert^2dx
\eeq
holds for almost any $t$. By integrating the above inequality, we obtain
\beq\label{eqn3.10}
\int_{B_{R_{0}}}\vert\nabla\varphi^\ast\vert^2 dx \leq
\int_{\Omega_{+}} \vert\nabla\varphi\vert^2 dx.
\eeq

On the other side, \eqref{eqn3.1} implies
$$
\int_{\Omega_{+}} \vert\nabla\varphi\vert^2 dx =
K\int_{\Omega_{+}}e^{v(x)}\ \varphi^2 dx,
$$
so that \eqref{eqn3.10} yields
$$
\int_{B_{R_{0}}} \vert\nabla\varphi^\ast\vert^2 dx -
\int_{B_{R_{0}}} e^{U(x)}(\varphi^\ast)^2 dx \leq
\int_{\Omega_{+}}\vert\nabla\varphi\vert^2 dx - \int_{\Omega_{+}}
e^{v(x)}\varphi^2 dx \leq 0.
$$
Thus, the first eigenvalue of $ \Delta + e^{U(x)} $ is
nonpositive. By a straightforward computation, the function
$z(r) = \frac{8 - r^2}{8 + r^2}$  satisfies
\beq\label{eqn 3.11}
\bigtriangleup z + e^{U(r)}z = 0 \quad \hbox{in} \quad
\mathbb{R}^2.
\eeq

Since $z(r)\geq0$ for $r \leq \sqrt{8}$, then the first eigenvalue of
$\bigtriangleup+e^{U(x)}$ for $B_{\sqrt{8}}$ is equal to zero. Hence\\
\beq\label{eqn3.12}
\sqrt{8} \leq R_{0},
\eeq\\
and \eqref{eqn3.2} readily follows since we have (see \eqref{eqn3.5})
$$
\int_{\Omega_{+}} e^{v(x)}dx = \int_{B_{R_{0}}} e^{U(x)}dx \geq
\int_{B_{\sqrt{8}}} e^{U(x)}dx =4\pi.
$$

Next, let $\Omega_{-}=\lbrace{x\vert \varphi(x)<0}\rbrace$. By using the same argument
we obtain
$$
\int_{\Omega_-} e^{v(x)}dx \geq 4\pi,
$$
and
$$
\rho=\int_{\Omega} e^{v(x)}dx = \int_{\Omega_{+}} e^{v(x)}dx +
\int_{\Omega{-}}e^{v(x)}dx \geq 8\pi.
$$

Of course, this is already a contradiction whenever $\rho<8\pi$.\\
In case $\rho=8\pi$ then (\ref{eqn3.2}) turns out to be an
equality i.e.
$$
\int_{\Omega_{+}} e^{v(x)}dx = 4\pi = \int_{\Omega_{-}} e^{v(x)}dx,
$$
\noindent and all the inequalities in (\ref{eqn3.8}) are
equalities. In particular, for $t>0$ and $s<0$  the domains
$\lbrace{x \in \Omega\vert\varphi(x)>t}\rbrace$ and $\lbrace{x \in
\Omega\vert\varphi(x)<s}\rbrace$ are simply-connected (by Theorem \ref{Thm2.1})
and in particular $\vert\bigtriangledown\varphi(x)\vert
e^{-v(x)}=\vert\bigtriangledown\varphi(y)\vert e^{-v(y)}$ whenever
$\varphi(x)=\varphi(y)\neq0$, because of the equality in the Cauchy-Schwarz inequality.\\
Therefore, passing to the limit, the equality above holds also for the set
$\{x\,|\,\varphi(x)=0\}$, i.e.,
\begin{equation}{\label{eqn3.13}}
|\nabla\varphi(x)|e^{-v(x)}=|\nabla\varphi(y)|e^{-v(y)}\,\,\hbox{
whenever }\,\,\varphi(x)=\varphi(y)=0.
\end{equation}
Since both $\Omega^+$ and $\Omega^-$ are simply-connected while
$\Omega$ is not, then the nodal line
$\{x\in\O\,|\,\varphi(x)=0\}$ must intersect $\partial\Omega$ at some point (say) $p_0\in\partial\O$. Clearly,
$\nabla\varphi(p_0)=0$, and (\ref{eqn3.13}) yields
\[
\nabla\varphi(x)=0\,\,\hbox{for}\,\,x\in\{x\in\overline{\O}\,|\,\varphi(x)=0\},
\]
which in view of the strong maximum principle implies $\varphi\equiv 0$. This is the desired
contradiction for $\rho=8\pi$ which concludes the proof of that part of $({\bf \mathrm{II}})$
which is concerned with $\lm_2\equiv\lm_2(e^v,\Omega)$.\\
At this point however the assertion concerning $\lm_1$ for $\rho<4\pi$ is easily worked out by arguing as we did
above via rearrangement and just replacing $\O^+$ with $\O$. Finally, in case $\rho=4\pi$, we conclude once more
that all the inequalities in \eqref{eqn3.8} are equalities, which is impossible in view of Theorem
\ref{Thm2.1} and the fact that $\O$ is multiply-connected.
\end{proof}

It turns out that the same argument used in the proof of $({\bf \mathrm{I}})$ in \cite{BL1}
as well as the one used in the proof of $({\bf \mathrm{II}})$ above
show that the following useful result holds:
\begin{lemma}\label{lem3.1.1}$\,$\\
Let either $\mathrm{w}$ satisfy \eqref{eqn1.6} and $\O$ be simply connected or
$v$ take the form \eqref{eqn2.5} and satisfy \eqref{eqn2.6} on
the multiply-connected domain $\O$ and set either $V=e^\mathrm{w}$ or $V=e^{v}$ respectively.\\
Assume moreover that $\int\limits_{\O}V\leq 8\pi$ and that on some subdomain
$\o\Subset \O$ there exists $\psi\in C^{2}_0(\o)\cap C^{0}(\ov{\o})$ which satisfies
$$
-\triangle \psi - V\psi\leq 0\quad\mbox{in}\quad \o.
$$
If $\psi>0$ in $\o$, then $\int\limits_{\o}V\geq 4\pi$.
\end{lemma}

Next, let us prove Theorem \ref{thm3.1}.

\textbf{The Proof of Theorem \ref{thm3.1}.}\\ The main point in the proof of Theorem
\ref{thm3.1} is to show that the linearized operator for (\ref{eqn2.1}) is non-singular whenever
$\rho\leq 8\pi$. Once this fact is known, the proof can be completed by known arguments,
see for example \cite{BL1}. We refer the reader to that paper for further details concerning this point.\\

We argue by contradiction and suppose that $\varphi$ is a solution of the linearized problem for \eqref{eqn2.1}
with $\rho\leq 8\pi$.
Then $\varphi$ satisfies
\begin{eqnarray*}
& & \D \varphi + \frac{\rho he^u\varphi}{\int_\O he^u } -
\frac{\rho he^u(\int_\O he^u \varphi)}
{(\int_\O he^u )^2}=0\ \ \mbox{in}\ \ \O,\\
& & \varphi\mid_{\p\O}=0.
\end{eqnarray*}
By adding $\frac{-\int he^u\varphi}{\int_\O he^u }$ to
$\varphi$ we come up with a new function (still denoted by $\varphi$) which satisfies
\begin{equation}{\label{eq:2-17}}
\left\{ \begin{array}{l}
\ds \D\varphi + \frac{\rho he^u \varphi}{\int_\O he^u } =0\ \ \mbox{in}\ \ \O,\\\\

\int_\Omega |\nabla\varphi|^2 =1\\\\

\ds \int_\O \rho he^u \varphi=0 \ \ \mbox{and $\varphi=c$ on
$\p\O$}, \end{array}\right.
\end{equation}
where $c$ is a constant. Since $\int\rho he^u\varphi=0$, $\varphi$
changes sign in $\O$. If $c=0$, then the second
eigenvalue of $\Delta+\frac{\rho he^u}{\int\rho he^u}$ would be
non-positive, in contradiction with Lemma \ref{lem3.1}$ ({\bf \mathrm{II}})$. Hence, we may assume that $c\neq 0$ and
in particular, without loss of generality, that in fact $c<0$.
We define
$$
\widetilde{\O}^+=\{x\in\O\,|\,\varphi(x)>c\},
$$
and divide the proof in two cases.

\noindent\textbf{Case 1:} $\widetilde{\O}^+=\Omega$.

We argue as in the proof of Lemma \ref{lem3.1} and apply the rearrangement
argument to obtain a contradiction.\\
Set $\O_t=\{x\in\O\,|\,\varphi(x)>t\}$ for
$t\in[c,\,\displaystyle\max_{\overline{\O}}\varphi]$ and $U(x)$ and
$\varphi^*$ as in Lemma \ref{lem3.1}. Thus, we have
\begin{eqnarray}\label{eqn3.14}
  &&\int_{\{\varphi^*>t\}}e^{U(x)}dx=\int_{\{\varphi>t\}}e^{v(x)}dx,\\
{\label{eqn3.15}}
&&\int_{B_{R_0}}e^{U(x)}\varphi^*dx=\int_\O e^{v(x)}\varphi dx=0,\\
\label{eqn3.16} &&\int_{B_{R_0}}|\nabla\varphi^*|^2<\int_\O
|\nabla\varphi|^2dx,
\end{eqnarray}
where $R_0=r(c)$ satisfies
\beq\label{eqn080612:1}
\int_{B_{R_0}}e^{U(x)}dx=\rho\leq8\pi,
\eeq
and $R_0<+\infty$ if $\rho<8\pi$, while $R_0=+\infty$ if
$\rho=8\pi$.\\
It is worth to point out that the strict inequality in \eqref{eqn3.16} is
due to the fact that $\O$ is not simply-connected (see Theorem \ref{Thm2.1}).\\
Since $c<0$ then is well defined $\xi_0=r(0)$ and clearly $\varphi^*(\xi_0)=0$. It is easy to verify
that we may assume \eqref{eqn3.2} (that is $\int_{\Omega_{+}}e^{v(x)}dx\geq4\pi$) to be satisfied in the situation
under consideration. Then we conclude that
$$
\int_{B_{\xi_0}}e^{U(x)}dx=\int_{\{\varphi^*>0\}}e^{U(x)}dx=\int_{\{\varphi>0\}}e^{v(x)}dx\geq4\pi.
$$
Hence,
\beq\label{080612:2}
\xi_0\geq\sqrt{8}.
\eeq
Next, we observe that
\beq\label{0506.3}
\int_\O |\nabla\varphi|^2dx=\int_\O e^v\varphi^2 dx=1.
\eeq
Therefore, putting
\begin{eqnarray*}
& & k_0=\inf \Biggl\{ \int_{B_{R_0}} |\bd \psi|^2dx \ \ | \ \
\psi(x)=\psi(|x|)\ \ \mbox{and}\ \
\psi(\xi_0)=0,\\
& & \quad \int_{B_{R_0}} e^{U(x)}\psi(x)dx=0,\ \ \
\int_{B_{R_0}} e^{U(x)} \psi^2 dx=1 \Biggl\},
\end{eqnarray*}
we see that  \eqref{0506.3} and \eqref{eqn3.6}, \eqref{eqn3.15}, \eqref{eqn3.16} together imply $k_0<1$.

It is not difficult to see that the infimum is always achieved by
some function $\psi^*$ where $\psi^*$ is continuous in
$\ov{B_{R_0}}$ and  satisfies
\begin{equation}{\label{eq:2-34}}
\left\{ \begin{array}{l} \D\psi^*+k_0 e^{U(x)}\psi^*=0\ \ \mbox{in}
\ \ 0< r<\xi_0\ \hbox{ and }\ \xi_0<r<R_0\\ \ds \int_{B_{R_0}}
e^{U(x)}\psi^*(x)dx=0,\ \psi^*(\xi_0)=0 \ \ \mbox{and}\ \
{\psi^*}'(R_0)=0.
\end{array}\right.
\end{equation}

If $R_0<+\infty$, then
\begin{eqnarray*}
& & \lim_{r\uparrow \xi_0} {\psi^*}'(r)r = -k_0\int_0^{\xi_0} e^{U(r)} \psi^*(r)rdr,\ \ \mbox{and}\\
& & \lim_{r\downarrow \xi_0} {\psi^*}'(r)r = k_0 \int_{\xi_0}^{R_0} e^{U(r)} \psi^*(r)rdr,
\end{eqnarray*}
the last equality being a consequence of the condition ${\psi^*}'(R_0)=0$. Since $\psi^*$ satisfies
$$
\int_{B_{R_0}} e^{U(r)} \psi^*(r) dx=0,
$$
we conclude that
\begin{equation}{\label{eq:2-35}}
\lim_{r\uparrow \xi_0} {\psi^*}'(r) = \lim_{r\downarrow \xi_0}
{\psi^*}'(r).
\end{equation}
If $R_0=+\infty$ and since ${\psi^*}'\in L^2(\R^2)$, there exists $r_n\to +\infty$ such that
$$
{\psi^*}'(r_n)r_n \to 0,
$$
which implies
$$
\lim_{r\downarrow \xi_0} {\psi^*}'(r)r = \int_{\xi_0}^\infty e^{U(x)} \psi^* (r)rdr.
$$
Therefore we readily verify that (\ref{eq:2-35}) holds in case $R_0=+\infty$ as well.
Hence $\psi^*$ is of class $C^2$ and satisfies
\begin{equation}\label{eq:2-35a}
\left\{ \begin{array}{l}
\D \psi^* + k_0 e^{U(x)} \psi^*=0\ \ \mbox{in $B_{R_0}$, and}\\
\ds \int_{B_{R_0}} e^{U(x)} \psi^* dx = 0 \ \ \mbox{and}\ \
\psi^*(\xi_0)=0,{\psi^*}^{'} (R_0)=0
\end{array}\right.
\end{equation}
for some $0<k_0<1$.

Clearly $\psi^*(r)$ changes signs once and \eqref{eqn080612:1} to be used together with
Lemma \ref{lem3.1.1} shows that in fact it changes sign just once.
Therefore we can assume without loss of generality that $\psi^*(r)>0$ if $r<\xi_0$ and $\psi^*(r)<0$ if
$r>\xi_0$. Hence, by using the integral constraint in \eqref{eq:2-35a},
we see that for each $0<r<R_0$ it holds
$$
\int_0^r e^{U(r)} \psi^*(r)rdr>\int_0^{R_0} e^{U(r)}\psi^*(r) rdr=0.
$$

Thus $\psi^{*'}(r)<0$ for $0<r<R_0$.

Let $z(r)=\frac{8-r^2}{8+r^2}$ be the function defined above and satisfying
\eqref{eqn 3.11}. Note that $z(\sqrt{8})=0$ and in particular that $\psi(\xi_0)=0$ and $\xi_0\geq
\sqrt{8}$ (see \eqref{080612:2}). We want to prove $\xi_0=\sqrt{8}$.

We assume (by contradiction) that $\xi_0>\sqrt{8}$. Then by using the equation in
\eqref{eq:2-35a} and \eqref{eqn 3.11} we see that for each $\xi_0<r<R_0$ it holds

\begin{eqnarray}{\label{eq:2-37}}
& & \lim_{R\to R_0} \Biggl( \frac{\psi^*(R)}{z(R)}\Biggl)' z^2(R)
- r\Biggl( \frac{\psi^*(r)}{z(r)}\Biggl)' z^2(r)\\ \nonumber &  &
=  (1-k) \int_r^{R_0} e^{U(s)} \psi^*(s) z(s) ds.
\end{eqnarray}

In the same time we see that either $R_0<+\infty$ and then
\beq\label{eq:2-38}
R_0 \Biggl( \frac{\psi^*(R_0)}{z(R_0)}\Biggl)' z^2(R_0)
=  R_0 (\psi^{*'} (R_0) z(R_0) - z'(R_0)\psi^*(R_0))
\eeq
$$
=  -R_0 z'(R_0) \psi^*(R_0)<0,
$$
or $R_0=+\infty$ and then
\begin{equation}{\label{eq:2-39}}
\lim_{R\to +\infty} R \Biggl( \frac{\psi^*(R)}{z(R)}\Biggl)' z^2
(R) = 0.
\end{equation}
Hence we can use \eqref{eq:2-37} together with \eqref{eq:2-38} and  \eqref{eq:2-39} to conclude that
$$
\Biggl( \frac{\psi^*(r)}{z(r)}\Biggl)'<0,\ \ \mbox{if}\ \xi_0<r<R_0.
$$
This inequality in turn yields
$$
0=\frac{\psi^*(\xi_0)}{z(\xi_0)} > \frac{\psi^*(r)}{z(r)}>0,\ \ \mbox{if}\ \sqrt{8}<\xi_0\leq r
$$
a contradiction. Therefore we conclude that $\xi_0\leq\sqrt{8}$ and in view of
\eqref{080612:2} $\xi_0=\sqrt{8}$ as desired.\\
At this point, by using \eqref{eq:2-35a}, we check that $\psi^*$ is a positive eigenfunction
for $\D + e^{U(r)}$ on the ball $B_{\sqrt{8}}$ corresponding to the eigenvalue $k_0-1$. Hence
$\psi^*$ must be the first eigenfunction corresponding to the first eigenvalue which, of course, is zero
(its eigenfunction being $z$).
Hence we must have $k_0=1$ which is the desired contradiction in Case 1.\\

\noindent {\bf Case 2.}\ \ \ We assume both $\widetilde{\O}^+=\{
x\in\O \mid \varphi(x)>c\}$ and $\widetilde{\O}^-=\{ x\in\O\mid
\varphi(x)<c\}$ are not empty sets.

In this case, we set $r(t)$ for $t>c$ and $R(t)$ for $t<c$ by
$$\int_{B_{r(t)}} e^{U(x)} dx = \int_{\{\varphi>t\}} e^{v(x)} dx \ \ \mbox{if}\ \ t>c,$$
and,
$$ \int_{\R^2\backslash B_{R(t)}} e^{U(x)} dx = \int_{\{\varphi<t\}} e^{v(x)} dx
\ \ \mbox{if}\ \ t<c.
$$
Let $r_0=\lim_{t\downarrow c} r(t)$ and $R_0=\lim_{t\uparrow c}
R(t)$. Of course we have
\beq\label{090612:1}
r_0\leq R_0\ \  \mbox{and} \ \ r_0=R_0 \ \ \mbox{if and only if} \ \
\rho=8\pi.
\eeq
Let $\varphi^*$ and $\tilde \varphi$ be the
symmetrization of $\varphi$ for the parts $\{\varphi>c\}$ and
$\{\varphi<c\}$ respectively. Thus, by arguing as in Lemma \ref{lem3.1}, we have
\begin{eqnarray*}
& & \int_{B_{r_0}} |\bd\varphi^*(x)|^2 dx + \int_{B_{r_0}} e^U |\varphi^*(x)|^2 dx \\
& \leq& \int_{\{ \varphi>c\}} |\bd\varphi(x)|^2 + \int_{\{\varphi>c\}} e^v |\varphi(x)|^2 dx,\\
\mbox{and} & & \\
& & \int_{\R^2\backslash B_{R_0}} |\bd\tilde \varphi(x)|^2 +
\int_{\R^2\backslash B_{R_0}}
e^U |\tilde \varphi(x)|^2 dx \\
& \leq & \int_{\{\varphi<c\}} |\bd\varphi(x)|^2dx +
\int_{\{\varphi<c\}} e^v |\varphi(x)|^2 dx.
\end{eqnarray*}
Therefore, by using \eqref{eq:2-17}, we obtain
\begin{eqnarray}{\label{eq:2-40}}
& & \int_{B_{r_0}} |\bd\varphi^*(x)|^2 + \int_{B_{r_0}} e^U
|\varphi^*(x)|^2 + \int_{\R^2\backslash B_{R_0}} |\bd \tilde
\varphi(x)|^2 \\ \nonumber & & + \int_{\R^2\backslash B_{R_0}} e^U
|\varphi^*(x)|^2 \leq 0.
\end{eqnarray}
Clearly, $\varphi^*$ and $\tilde \varphi$ together satisfies
$$\int_{B_{r_0}} e^U \varphi^*(x) + \int_{\R^2\backslash B_{R_0}} e^U \tilde \varphi(x) dx
= \int_\O e^v \varphi(x) dx =0.$$ Let $\xi_0=r(0)<r_0=r(c)$. Then
$\varphi^*(\xi_0)=0$ and we may still assume without loss of generality \eqref{eqn3.2}
(that is $\int_{\Omega_{+}}e^{v(x)}dx\geq4\pi$) to be satisfied so that also
\eqref{080612:2} (that is $\xi_0\geq\sqrt{8}$) holds in Case 2 as well.
\begin{eqnarray*}
\mbox{Set}
& & H=\Biggl\{(\psi^*,\tilde\psi)\mid \psi^*(x)=\psi^*(|x|)\ \ \mbox{and}\ \ \tilde\psi(x)=\tilde\psi(|x|)\\
& & \hbox{are defined in}\ \ B_{r_0}\ \ \hbox{and}\ \ \mathbb{R}^2\backslash B_{R_0}\ \ \hbox{respectively, and satisfy}\\
& & \psi^*(\xi_0)=0, \psi^*(r_0)=\tilde\psi(R_0)=c\ \ \hbox{and}\\
& & \int_{B_{r_0}} e^U \psi^*(x) dx + \int_{\R^2\backslash
B_{R_0}} e^U \tilde\psi(x)dx=0\Biggl\},
\end{eqnarray*}

and
$$k=\inf \left\{\int_{B_{r_0}} |\bd\psi^*|^2 + \int_{\R^2\backslash B_{R_0}} |\bd\tilde\psi|^2dx\right\},$$
where the infimum is taken over the set of all $(\psi^*,\tilde \psi)\in H$ such that
$$\int_{B_{r_0}} e^U |\psi^*|^2 dx + \int_{\R^2\backslash B_{R_0}} e^U |\tilde \psi(x)|^2 dx
=1.$$
Hence, by using (\ref{eq:2-40}) we obtain
$$
0<k\leq 1.
$$
It is easy to see that the infimum is achieved and we denote by
$(\psi^*, \tilde \psi)$ the corresponding minimizers. By using the same arguments adopted above,
we can verify that \eqref{eq:2-35} holds in Case 2 as well. In particular we have
\begin{equation}{\label{eq:2-41}}
\left\{ \begin{array}{l}
\D\psi^* + k e^U \psi^*=0\ \ \mbox{in} \ \ B_{r_0},\\
\D\tilde \psi + k e^U \tilde\psi=0\ \ \mbox{in} \ \ \R^2\backslash B_{R_0},\\
\psi^*(\xi_0)=0, \psi^*(r_0)=\tilde \psi(R_0)=c\ \ \mbox{and}\\
\ds\int_{B_{r_0}} e^U \psi^* dx + \int_{\R^2\backslash B_{R_0}}
e^U \tilde \psi(x) dx =0
\end{array}\right.
\end{equation}
Since ${\tilde \psi}' (r) r \to 0$ as $r\to +\infty$, we also have
\beq\label{eq:2-42}
{\psi^*}'(r_0)r_0={\psi^*}' (R_0) R_0.
\eeq
At this point we want to show $\xi_0=\sqrt{8}$. Set
$z(r)=\frac{r^2-8}{r^2+8}$.

Hence, let us assume (by contradiction) that $\xi_0>\sqrt{8}$. Clearly we can use \eqref{eq:2-37}
to obtain
$$
r_0\left(\frac{\psi^*(r_0)}{z(r_0)}\right)^{'} z^2(r_0) -
\xi_0\left(\frac{\psi^*(\xi_0)}{z(\xi_0)}\right)^{'} z^2(\xi_0)
$$
$$
=(1-k)\int^{r_0}_{\xi_0} e^{U(s)}\psi^*(s)z(s)ds\geq 0,
$$
where $\psi^*(s)\leq 0$ and $z(s)\leq 0$ for
$s\geq\xi_0\geq\sqrt{8}$.\\
\noindent Thus,
$r_0({\psi^*}^{'}(r_0)z(r_0)-z'(r_0)\psi^*(r_0))\geq\xi_0{\psi^*}^{'}(\xi_0)z(\xi_0)>0$,
where ${\psi^*}^{'}(\xi_0)<0$ and $z(\xi_0)<0$. Therefore, since $z(r_0)\psi^*(r_0)>0$, we
readily obtain
$$
\frac{{\psi^*}'(r_0)}{\psi^*(r_0)} > \frac{z'(r_0)}{z(r_0)}.
$$
This relation can be used together with \eqref{eq:2-42} to conclude that
\beq\label{eq:2-43}
\frac{{\psi^*}'(R_0)R_0}{\psi^*(R_0)}=
\frac{{\psi^*}'(r_0)r_0}{\psi^*(r_0)}> \frac{z'(r_0)}{z(r_0)} r_0.
\eeq
One the other hand, by a straightforward computation, we have
$$
\frac{z'(r)}{z(r)} = \frac{2r}{r^2-8} - \frac{2r}{r^2+8} = \frac{32r}{r^4-64},
$$
and then
$$
\frac{z'(r)r}{z(r)} = \frac{32r^2}{r^4-64}\ \ \mbox{is decreasing for}\ \ r>\sqrt{8}.
$$
We use this fact together with \eqref{eq:2-43} to obtain
\beq\label{eqn3.32}
 \frac{{\psi^*}'(R_0)R_0}{\psi^*(R_0)} >
\frac{z'(r_0)r_0}{z(r_0)} \geq \frac{z'(R_0)R_0}{z(R_0)},
\eeq
which in turn implies
$$
{\psi^*}'(R_0)z(R_0) > z'(R_0) \psi^* (R_0).
$$
This inequality contradicts the fact that, by using the second equation in \eqref{eq:2-41},
we have
$$
0>-R_0\Biggl(\frac{\psi^*(R_0)}{z(R_0)}\Biggl)' z^2(R_0)=(1-k) \int_{R_0}^\infty
e^{U(s)} \tilde \psi (s) z(s)sds\geq 0.
$$
Thus $\xi_0\leq\sqrt{8}$ and since the reversed inequality holds as well, then $\xi_0=\sqrt{8}$.
As in Case 1 we conclude that $k=1$ and then repeat the argument starting with \eqref{eq:2-42} to
conclude that indeed the second inequality
in \eqref{eqn3.32} must be an equality. Hence $r_0=R_0$ and \eqref{090612:1} shows that $\rho=8\pi$.
In particular, as in Case 1, all the inequalities used in the rearrangement argument are equalities.
Therefore it follows once more from Theorem \ref{Thm2.1} that both $\tilde{\Omega}^+$ and $\tilde{\Omega}^-$
are simply-connected and from the Cauchy-Schwarz inequality that $|\nabla\varphi(x)|=|\nabla\varphi(y)|$
whenever $\varphi(x)=\varphi(y)=c$. Since $\Omega$ is not
simply-connected, there exists a point $P_0\in\partial\Omega$ such
that $\varphi(P_0)=c$ and $\nabla\varphi(P_0)=0$ which in turn yields
$\nabla\varphi(x)=0$ for $x\in\partial\Omega$. This fact clearly contradicts the strong maximum principle
and therefore concludes the proof of Theorem \ref{thm3.1}.
\end{proof}

\section{Existence of a solution at $\rho=8\pi$}\label{sec4}
\setcounter{equation}{0}

Solutions of \eqref{eqn2.1} are critical points of the functional $I_\rho:H_0^1(\O)\mapsto \R$
defined as
\beq\label{eqn4.1}
I_\rho(u)=\frac{1}{2}\int_\O|\nabla u|^2-\rho\log\int_\O h(x)e^u .
\eeq
A well known consequence of the Moser-Trudinger \cite{moser} inequality is that a minimizer $u_\rho$ of
$I_\rho$ exists at least in case $\rho<8\pi$. Actually any such minimizer for $\rho<8\pi$ is the unique solution to
\eqref{eqn2.1} according to the uniqueness theorem. In particular, by using the
implicit function theorem and the invertibility of the linearized equation at $\rho=8\pi$
we have (see for example Proposition 6.1 in \cite{key02} for a proof):

\begin{lemma}\label{lem100612.1}
The following facts are equivalent:\\
(i)   $u_\rho$ converges in $C^2(\overline{\Omega})$ as $\rho\nearrow 8\pi$,\\
(ii)  A subsequence $u_{\rho_n}$ converges in $C^2(\overline{\Omega})$ as $\rho_n\nearrow 8\pi$,\\
(iii) Equation (2.1) possesses a solution at $\rho=8\pi$.\\
(iv)  $I_{8\pi}$ attains its infimum.
\end{lemma}

Let $\widetilde{G}(x,\,p)$ and $\gamma(p)$ be defined in \eqref{eqn1.2}.
We consider the situation where the (unique) branch of minimizers contains
a sequence of blowing up solutions, say $u_{\rho_n}$, as $\rho_n\nearrow8\pi$.
Let $q$ be a blowup point of $u_{\rho_n}(x)$, i.e., there exists a subsequence $u_k\equiv u_{\rho_{n_k}}$ such that
\begin{center}
$u_k(q_k)=\max\limits_{\overline{\O}}u_{\rho_k}(x)\rightarrow
+\infty$ and $q_k\rightarrow q$.
\end{center}

\begin{remark}\label{rem:qmax}
In this situation it can be shown that $q$ is a maximum point of $\,\log h(x)+4\pi\gamma(x)$ and in particular that
\beq\label{eqn4.2}
\frac{1}{8\pi}\inf_{u\in H^1_0(\O)} I_\rho(u) =
-1-\log(\pi)-\sup_{\overline{\Omega}}(h(x)+4\pi\gamma(x)).
\eeq
These facts are well known. See for example Theorem 1.2 and Lemma 2.3 in \cite{BL2} for a proof.
Actually, $q$ is the \un{unique} blow up point of $u_\rho$ as $\rho\nearrow 8\pi$ as shown in the following
Proposition \ref{prop4.1}.
\end{remark}

\begin{Proposition}\label{prop4.1}
Suppose that a sequence of blowing up solutions, denoted by $u_{\rho_n}$, exists for \eqref{eqn2.1}
as $\rho_n\nearrow 8\pi$ and let $q$ be a blow up point. Then,
$u_\rho(x)$ converges to $8\pi G(x,q)$ in $C^2_{loc}(\overline{\Omega}\setminus{q})$ as
$\rho\nearrow 8\pi$.
\end{Proposition}
\begin{proof}
Suppose $q'\neq q$ is a blow up point of $u_{\rho_n'}$. Without
loss of generality, we may assume $\rho_{n-1}'<\rho_n<\rho_n'$.
Let $\delta>0$ such that $\delta<\frac{1}{2}\,\mbox{dist}(q,q')$ and
$$
m_\rho=\dfrac{\sup_{B(q,\delta)}u_\rho(x)}{\sup_{\overline{\Omega}}u_\rho(x)}.
$$
\noindent Obviously, $m_\rho$ continuously depends on $\rho$.\\
Since $u_{\rho_n}$ and $u_{\rho_n'}$ blows up at $q$ and $q'$
respectively, we have for large $n$,
$$
m_{\rho_n}=1\ \ \hbox{and}\ \ m_{\rho_n'}=o(1).
$$
Thus, there exists $\rho_n''\in[\rho_n,\rho_n']$ such that
$m_{\rho_n''}=\frac{1}{2}$. Obviously, as $n\rightarrow +\infty$,
$\rho_n''\rightarrow 8\pi$ and
$\sup_{\overline{\Omega}}u_{\rho_n''}\rightarrow +\infty$. By our
choice $m_{\rho_n''}=\frac{1}{2}$, $u_{\rho_n''}$ has at least two
blowup points, which is impossible in view of by now standard
concentration-compactness results \cite{bm}, \cite{yy} for Liouville-type equations.\\
Therefore we conclude that any blow up sequence extracted from $u_\rho$ as $\rho\nearrow 8\pi$
admits $q$ as its unique blow up point. It is well known that
blow up points are necessarily interior points (see Lemma 2.1 in \cite{key02}). Thus the results in
\cite{yy} apply and we conclude that any such sequence must in fact converge
to $8\pi G(x,q)$ in $C^2_{loc}(\overline{\Omega}\setminus{q})$ and we conclude in particular, in view of
the equivalence of Lemma \ref{lem100612.1} above, that the full branch of minimizers $u_\rho$ satisfies to the
same property as well.
\end{proof}

The following asymptotic estimates for
$\rho_n-8\pi$ along a blowing up sequence $u_{\rho_n}$ was obtained in \cite{key02} and \cite{key05}:

\noindent(j) If $\Delta\log h(q)\neq 0$, then
    \begin{eqnarray}\label{eqn4.3}
    \rho_n-8\pi=c(\Delta\log h(q)+o(1))\lambda_ne^{-\lambda_n}
    \end{eqnarray}
\noindent(jj) If $\Delta\log h(q)=0$, then
    \begin{equation}\label{eqn4.4}
    \rho_{n}-8\pi=h(q)(D_h(q)+o(1))e^{-\lambda_{n}},
    \end{equation}
    where $\lambda_{n}=\max_{\overline{\Omega}}u_{n}-\log{(\int_{\Omega}h(x)e^{u_{n}}dx})$,
    $o(1)\rightarrow 0$ as $n\rightarrow +\infty$, and
    \begin{equation}\label{eqn4.5}
    D_h(q)=\lim_{\varepsilon\rightarrow0}\int_{\Omega\setminus B(q,\varepsilon)}
    \dfrac{\frac{h(x)}{h(q)}\,e^{8\pi(\widetilde{G}(x,q)-\gamma(q))}}{\vert x-q \vert^4} dx - \int_{\Omega^c} \frac{dx}{\vert x-q
    \vert^4}.
    \end{equation}

\bigskip

\begin{remark}\label{rem:CPAM}
By using the results in either \cite{EGP} or \cite{MMM} we see that a sequence of solutions for \eqref{eqn2.1}
blowing up as $\rho\to 8\pi$ can be constructed whenever $q$ is a nondegenerate critical point of
$\log h(x)+4\pi\gamma(x)$. Alternatively, by using the condition $D_h(q)\neq 0$, a sequence of solutions for \eqref{eqn2.1}
blowing up as $\rho\to 8\pi$ can be constructed by arguing as in \cite{key06}. In any case, if $q$ is a blow up point,
then $h(q)\neq 0$ by known blow up arguments, see \cite{lisha}.
\end{remark}

\begin{theorem}\label{thm4.1}
Let $\Omega$ be a $C^1$ bounded domain and $h(x)$ satisfy (2.2).
Then equation \eqref{eqn2.1} at $\rho = 8\pi$ admits a solution
if and only if there exists a maximum point $q$ of
$\log{h(x)}+4\pi\gamma(x)$ such that $D_h(q)>0$.
\end{theorem}

\begin{proof}

We first prove that the condition is sufficient. Suppose $D_h(q)>0$ for a maximum
point $q$ which we can assume without loss of generality to coincide with
the origin $q=0$.\\
We argue by contradiction and suppose that no solutions exist for $\rho = 8\pi$. In view of Lemma \ref{lem100612.1}
we see that necessarily a sequence of blowing up solutions can be found as $\rho_n \nearrow 8\pi$. As observed above,
in this situation \eqref{eqn4.2} holds.
We can assume without loss of generality that $B_1\subset\subset \Omega$ and then define
\begin{eqnarray*}
v_\varepsilon(x)=
\left\{%
\begin{array}{ll}
    4\log\frac{1}{|x|} + 8\pi\widetilde{G}(x,0) \quad  \hbox{for} \quad |x|\geq 1, \\
    2\log\left(\frac{\varepsilon^2+1}{\varepsilon^2+|x|^2}\right) + 8\pi\widetilde{G}(x,0) \quad  \hbox{for}
\quad |x|\leq 1,\\
\end{array}%
\right.
\end{eqnarray*}
so that in particular $v_\ve\equiv 8\pi G(x,q)=8\pi G(x,0)$ in $\Omega\setminus B_1$. Then we obtain the following

\begin{lemma}\label{lem25.06}
It holds,
$$
I_{8\pi}(v_\varepsilon)=-8\pi-8\pi\log{\pi}-8\pi(\log (h(0)) + 4\pi\gamma(0)) -
8\pi\left(\frac{D_h(0)}{\pi}\right)\varepsilon^2+{O}(\ve^3).
$$
\end{lemma}

\begin{proof}
$$
\frac{1}{2}\int_\Omega |\nabla v_\varepsilon|^2 dx =
\frac{1}{2}\int_{\Omega\setminus
B_1}\nabla\left(\log\frac{1}{|x|^4}+8\pi\widetilde{G}(x,0)\right)\nabla\log\frac{1}{|x|^4}dx
$$

$$
+4\pi\left(\int_{\Omega\setminus
B_1}\nabla\log{\frac{1}{|x|^4}}\nabla\widetilde{G}(x,0)dx
+8\pi\int_{\O}\nabla\widetilde{G}(x,0)\nabla\widetilde{G}(x,0)dx\right)
$$

$$
+\frac{1}{2}\int_{B_1}\left|\frac{4|x|}{\varepsilon^2+|x|^2}\right|^2 dx+
8\pi\int_{B_1}\nabla\left(2\log\left(\frac{\varepsilon^2+1}{\varepsilon^2+|x|^2}\right)\right)\nabla\widetilde{G}(x,0)
$$

$$
=\frac{-1}{2}\int_{\partial B_1}
\left(\log\frac{1}{|x|^4}+8\pi\widetilde{G}(x,0)\right)
\frac{\partial}{\partial\nu}\left(\log\frac{1}{|x|^4}\right)d\sigma
$$

$$
-4\pi\int_{\partial B_1}\left(\log\frac{1}{|x|^4}\right)\frac{\partial\widetilde{G}}
{\partial\nu}(x,0)d\sigma-8\pi-8\pi\log\left(\ve^2\right)+16\pi\ve^2+ \mbox{O}(\ve^4)
$$
$$
+8\pi\int_{\partial B_1}
\left(2\log\left(\frac{\varepsilon^2+1}{\varepsilon^2+|x|^2}\right)\right)\frac{\partial\widetilde{G}}
{\partial\nu}(x,0)d\sigma
$$

$$
= -16\pi\log\left(\ve\right) -8\pi+ 32\pi^2 \gamma(0)+16\pi\ve^2+ \mbox{O}(\ve^4),
$$
where $\nu$ denotes the exterior unit normal. We used here the fact that, since $\widetilde{G}(x,0)$ is harmonic, then
$$
\frac{1}{2\pi} \int\limits_{\partial B_1}\widetilde{G}(x,0)d\sigma
=\widetilde{G}(0,0)\equiv \gamma(0),
$$
and in particular
$$
\int_{\partial B_1}\frac{\partial\widetilde{G}}{\partial\nu}(x,0)d\sigma\equiv
\int_{\partial \O}\frac{\partial\widetilde{G}}{\partial\nu}(x,0)d\sigma=0.
$$
Observe moreover that, since $q=0$ is a critical point of $\log(h)+8\pi \widetilde{G}$, we have
\beq\label{2606.1.150812}
h(x)e^{8\pi\widetilde{G}(x,0)}=h(0)e^{8\pi\gamma(0)}+\sum\limits_{i,j=1}^{2} b_{ij}x_i x_j+O(|x|^3),
\eeq
where $b_{11}+b_{22}=0$ because $\log{(h)}+8\pi \widetilde{G}$ is harmonic in $\O$. By using
\eqref{2606.1.150812} and
some straightforward evaluation, we can conclude that
$$
\int_{B_1\setminus B_\ve} h(x)e^{8\pi\widetilde{G}(x,0)}
\left[\frac{(\varepsilon^2+1)^2}{(\varepsilon^2+|x|^2)^2}-\frac{1}{|x|^4}\right]dx=
$$
$$
-h(0)e^{8\pi\gamma(0)}\frac{\pi(\ve^2-1)^2}{2\ve^2}+\mbox{O}(\varepsilon)
=h(0)e^{8\pi\gamma(0)}\left[-\frac{\pi}{2\ve^2}+\pi+\mbox{O}(\varepsilon)\right],
$$
and
$$
\int_{\Omega\setminus B_\ve}\dfrac{h(x)e^{8\pi \widetilde{G}(x,0)}}{|x|^4}\,dx=
h(0)e^{8\pi\gamma(0)}\int_{\Omega \setminus B_\ve}\dfrac{\frac{h(x)}{h(0)}\,e^{8\pi(\widetilde{G}(x,0)-\gamma(0))}\;-1}
{|x|^4}\,dx
$$
$$
+h(0)e^{8\pi\gamma(0)}\int_{\mathbb{R}^2 \setminus
B_\varepsilon}\frac{dx}{|x|^4}
-h(0)e^{8\pi\gamma(0)}\int_{\mathbb{R}^2\setminus\Omega}\frac{dx}{|x|^4}
$$
$$
=h(0)e^{8\pi\gamma(0)}\left[D_h(0)+\frac{\pi}{\ve^2}+\mbox{O}(\varepsilon)\right],
$$
and
$$
\int_{B_\ve} h(x)e^{8\pi\widetilde{G}(x,0)}\frac{(\ve^2+1)^2 }{(\ve^2+|x|^2)^2}dx =
h(0)e^{8\pi\gamma(0)}\frac{\pi(\ve^2+1)^2}{2\ve^2}+ \mbox{O}(\varepsilon)
$$
$$
=h(0)e^{8\pi\gamma(0)}\left[\frac{\pi}{2\ve^2}+\pi + \mbox{O}(\varepsilon)\right].
$$
Therefore we have
$$
\int_\Omega h(x)e^{u_\varepsilon(x)}dx
=\int_{\Omega\setminus B_1}\dfrac{h(x)e^{8\pi \widetilde{G}(x,0)}}{|x|^4}\,dx+
\int_{B_1} h(x)e^{8\pi\widetilde{G}(x,0)}\frac{(\varepsilon^2+1)^2}{(\varepsilon^2+|x|^2)^2}dx
$$
$$
=\int_{\Omega\setminus B_\ve}\dfrac{h(x)e^{8\pi \widetilde{G}(x,0)}}{|x|^4}\,dx-
\int_{B_1\setminus B_\ve}\dfrac{h(x)e^{8\pi \widetilde{G}(x,0)}}{|x|^4}\,dx
$$
$$
+\int_{B_1\setminus B_\ve} h(x)e^{8\pi\widetilde{G}(x,0)}\frac{(\varepsilon^2+1)^2}{(\varepsilon^2+|x|^2)^2}dx+
\int_{B_\ve} h(x)e^{8\pi\widetilde{G}(x,0)}\frac{(\varepsilon^2+1)^2}{(\varepsilon^2+|x|^2)^2}dx
$$
$$
=h(0)e^{8\pi\gamma(0)}\left[D_h(0)+\frac{\pi}{\ve^2}-\frac{\pi}{2\ve^2}+\pi+\frac{\pi}{2\ve^2}+\pi+\mbox{O}(\varepsilon)\right]
$$
$$
=h(0)e^{8\pi\gamma(0)}\frac{\pi}{\ve^2}\left[1+\frac{D_h(0)+2\pi}{\pi}\ve^2  +\mbox{O}(\ve^3) \right].
$$

At this point we may collect together the above estimates to conclude that
$$
I_{8\pi}(v_\varepsilon)=\frac{1}{2}\int_\Omega|\nabla
v_\varepsilon|^2 -8\pi\log\int_\Omega h(x)e^{v_\varepsilon(x)}
$$

$$
=-16\pi\log(\ve)-8\pi +32\pi^2\gamma(0)+16\pi\ve^2+\mbox{O}(\ve^4)
$$

$$
-8\pi\log\left(\frac{\pi}{\ve^2} h(0)e^{8\pi\gamma(0)}\right)-8\pi \log\left(1+\frac{D_h(0)+2\pi}{\pi}\ve^2  +\mbox{O}(\ve^3) \right)
$$
$$
=-8\pi-8\pi\log\left(\pi\right)
-8\pi\left(\log(h(0))+4\pi \gamma(0)\right)+16\pi\ve^2-8\pi\frac{D_h(0)+2\pi}{\pi}\,\ve^2  +\mbox{O}(\ve^3)
$$
$$
-8\pi-8\pi\log\left(\pi\right)
-8\pi\left(\log(h(0))+4\pi \gamma(0)\right)-8\pi\frac{D_h(0)}{\pi}\,\ve^2+\mbox{O}(\ve^3).
$$

\end{proof}

The expansion provided by Lemma \ref{lem25.06} can be used together with the assumption
$D_h(0)>0$ to obtain
$$
\inf\limits_{u\in H^1_0(\O)} I_{8\pi}(u) < -8\pi -8\pi\log(\pi)-8\pi(\log h(0)+4\pi\gamma(0))
$$
$$
\qquad\qquad=-8\pi -8\pi\log(\pi)-8\pi\displaystyle\sup_{\overline{\Omega}}(\log h(x)+4\pi\gamma(x)),
$$
which is in contradiction with \eqref{eqn4.2}. Hence a solution exists and the sufficiency of the condition is proved.\\

Next, let us prove the necessary part and suppose that \eqref{eqn2.1} at $\rho=8\pi$ admits a solution.
We want to prove in this situation a stronger result, that is, $D_h(q)>0$ for any maximum
point $q$ of $\log h(x)+4\pi\gamma(x)$. By contradiction we assume that a maximum point $q_0$ exists such that
$D_h(q_0)\leq 0$.
The following Lemma \ref{lem4.1} shows in this case that $u_\rho$
blows up as $\rho\nearrow 8\pi$. This is of course in contradiction with Lemma \ref{lem100612.1} and
we may conclude that indeed $D_h(q)>0$ for all maximum points.
\end{proof}

\begin{lemma}\label{lem4.1}
Suppose that $q$ is a critical point of $\log h(x)+4\pi\gamma(x)$
and $D_h(q)\leq0$. Then there exists a sequence of solutions $u_n$ of
\eqref{eqn2.1} with $\rho_n<8\pi$ and $\rho_n\nearrow8\pi$ such
that $u_n$ blows up at $q$.
\end{lemma}

\begin{proof}

The proof will be divided in two cases.

\item{\textbf{Case 1:}} $D_h(q)<0$.

If $q$ is a nondegenerate critical point of $\log
h(x)+4\pi\gamma(x)$, then by Remark \ref{rem:CPAM} we can construct
a sequence of blowing-up solutions $u_n$ of \eqref{eqn2.1} with $\rho=\rho_n$
whose unique blow up point is $q$ (see Proposition \ref{prop4.1}). Of course $\rho_n<8\pi$
because of \eqref{eqn4.4} and $D_h(q)<0$.

If $q$ is a degenerate critical point
we assume without loss of generality that
\begin{eqnarray*}
\frac{\partial^2}{\partial x_1\partial x_2}(\log
h(x)+4\pi\gamma(x))=0 \quad \hbox{at} \quad x=q.
\end{eqnarray*}
Then we let
$$
h_\varepsilon(x)=h(x)\exp\left(\varepsilon[(x_1-q_1)^2-(x_2-q_2)^2]\right),
$$
where $q=(q_1,\,q_2)$. It is easy to see that $h_\varepsilon(x)$
satisfies \eqref{eqn2.2}, and $q$ is a nondegenerate critical
point of $\log h_\varepsilon(x)+4\pi\gamma$ for any
$\varepsilon>0$. Let $D_\varepsilon(q)$ be the quantity defined in
 \eqref{eqn4.5}, where $h(x)$ is replaced by $h_\varepsilon(x)$.
Obviously for small $\varepsilon>0$ we have $D_\varepsilon(q)<0$.

Let \eqref{eqn2.1}$_\varepsilon$ denote problem \eqref{eqn2.1}
where $h(x)$ has been replaced by $h_\varepsilon(x)$. Since
$h_\varepsilon(x)$ satisfies \eqref{eqn2.2}, then
\eqref{eqn2.1}$_\varepsilon$ admits a unique solution
$u_\rho^\epsilon(x)$ for any $\rho<8\pi$. Since $q$ is a nondegenerate critical point
of $\log h_\varepsilon+4\pi\gamma$
then we can construct a sequence of solutions for \eqref{eqn2.1}$_\varepsilon$ which blows
up at $q$ (see Remark \ref{rem:CPAM}). Thus, by using \eqref{eqn4.4}, the fact that $D_\varepsilon(q)<0$
and the uniqueness theorem, we conclude that $u_\rho^\epsilon(x)$
coincides with this sequence and hence blows up as $\rho\nearrow 8\pi$.

Let $C$ be a fixed large positive number and $\delta$ be a small
positive number. Then for each $\varepsilon > 0$, there exists a
$\rho^\varepsilon \in (0,8\pi)$ such that the solution
$u_\varepsilon = u_{\rho}^\varepsilon$ of
(\ref{eqn2.1})$_\varepsilon$ with $\rho=\rho^\varepsilon$
satisfies
\begin{equation}\label{eqn4.6}
\max_{\overline{\Omega}}u_\varepsilon(x)= C\quad \mbox{and}\ \
\sup_{B(q,\delta)}u_\varepsilon(x) =  C\quad \geq \quad 2
\sup_{\overline{\Omega} \setminus B(p,\delta)}u_\varepsilon(x).
\end{equation}
By letting $\varepsilon\rightarrow 0$, and $\rho^\varepsilon
\rightarrow\rho(C) \in (0,8\pi]$, there exists a solution $u(x;C)$
of (\ref{eqn2.1}) such that (\ref{eqn4.6}) holds. Clearly Theorem \ref{thm3.1} implies
that $u(x;C_1)\neq u(x;C_2)$ and $\rho(C_1)\neq\rho(C_2)<8\pi$ whenever $C_1\neq C_2$.
As $C\rightarrow +\infty$ we obtain a sequence of solutions for \eqref{eqn2.1} which
blows up at $q$ which is the desired conclusion in \textbf{Case 1}.

\item{\textbf{Case 2:}} $D_h(q)=0$

For $0<t<1$ we define
$$
d(t)=\int_{\Omega} \dfrac{\left(\frac{h(x)}{h(q)}\right)^t e^{8\pi
t(\widetilde{G}(x,q)-\gamma(q) )}-1}{\vert x-q \vert^4}\, dx -
\int_{\Omega^c} \dfrac{dx}{\vert x-q \vert^4}.
$$
Clearly,
\begin{equation}\label{4.7}
d'(t)=\int_{\Omega}\dfrac{\left[\log{\left(\frac{h(x)}{h(q)}\right)}+8\pi(\widetilde{G}(x,q)-\gamma(q))\right]
\left(\frac{h(x)}{h(q)}\right)^t
e^{8\pi t(\widetilde{G}(x,q)-\gamma(q))}}{\vert x-q \vert^4}\, dx
\end{equation}

and

\begin{equation}\label{4.8}
d''(t)=\int_{\Omega}\dfrac{\left[\log{\left(\frac{h(x)}{h(q)}\right)}+8\pi(\widetilde{G}(x,q)-\gamma(q))\right]^2
\left(\frac{h(x)}{h(q)}\right)^t
e^{8\pi t(\widetilde{G}(x,q)-\gamma(q))}}{\vert x-q \vert^4}\, dx.
\end{equation}

Since $\log\left(\frac{h(x)}{h(q)}\right)+8\pi(\widetilde{G}(x,q)-\gamma(q))$ is harmonic and
$\log\left(\frac{h(x)}{h(q)}\right) + 8\pi(\widetilde{G}(x,q)-\gamma(q))=O(\vert x-q \vert^2)$,
then the integral defining $d'(t)$ is well-defined in the sense of the following limit
$$
\lim_{\varepsilon\rightarrow 0} \int_{\Omega\setminus
B(q,\varepsilon)}
\dfrac{\left[\log{\left(\frac{h(x)}{h(q)}\right)}+8\pi(\widetilde{G}(x,q)-\gamma(q))\right]\left(\frac{h(x)}{h(q)}\right)^t
e^{8\pi t(\widetilde{G}(x,q)-\gamma(q))}}{\vert x-q \vert^4}\, dx.
$$

Since $d(1)=D_h(q)=0$ and $d(0) = -\int_{\Omega^c}\frac{dx}{\vert x-q \vert^4} < 0 $, by using $d''(t)>0$,
we have $d(t)<0$ for $t\in [0,1)$, that is

\beq\label{4.9}
\int_{\Omega} \dfrac{\left(\frac{h(x)}{h(q)}\right)^t e^{8\pi t(\widetilde{G}(x,q)-\gamma(q) )}-1}{\vert x-q \vert^4}\,dx
- \int_{\Omega^c} \dfrac{dx}{\vert x-q \vert^4} < 0.
\eeq

Let
$h_\epsilon(x)=(h(x))^{1-\epsilon}e^{-8\pi\epsilon\widetilde{G}(x,q)}$.
Then $q$ is a critical point of $\log
h_\epsilon(x)+8\pi\widetilde{G}(x,q)$, and
$$
D_\epsilon(q)=\int_{\Omega}
\frac{\frac{h_\epsilon(x)}{h_\epsilon(q)}\,e^{8\pi(\widetilde{G}(x,q)-\gamma(q))}-1}{|x-q|^4}\,dx
-\int_{\Omega^c}\frac{dx}{|x-q|^4}=d(1-\epsilon)<0.
$$
Now we consider
\begin{eqnarray}\label{eqn4.10}
\left\{%
\begin{array}{ll}
    \Delta u^\epsilon + \rho\frac{h_\epsilon(x)e^{u^\epsilon}}{\int_{\O}h_\epsilon(x)e^{u^\epsilon}dx}
& \hbox{in} \quad\O \\
    u^\epsilon=0 & \hbox{on} \quad \partial\O \\
\end{array}%
\right.
\end{eqnarray}
Note that
$$
\log
{h_\epsilon(x)}=(1-\epsilon)\log(h(x))-8\pi\epsilon\widetilde{G}(x,q)
$$
$$
\qquad\qquad\qquad\qquad\qquad
=(1-\epsilon)\log{h(x)}-8\pi\epsilon G(x,q)-4\epsilon\log|x-q|.
$$
By \eqref{eqn2.2} $\log h(x)$ can be extended to $\Omega^*$ as a
subharmonic function. As above, $G(x,q)$ can also be extended to
$\Omega^*$ by setting
$$
    \widehat{G}(x,q)=\left\{%
\begin{array}{ll}
    G(x,q) & \hbox{ if }\,x\in\Omega, \\
    0 & \hbox{ if }\,x\in\Omega^*\setminus\Omega. \\
\end{array}%
\right.
$$

Thus, we can extend $\log h_\epsilon(x)$ to the larger domain
$\Omega^*$. However, the extended function $-\widehat{G}(x,q)$ is
not subharmonic in $\Omega^*$. Hence $\log h_\epsilon(x)$ does not
satisfy \eqref{eqn2.2} and therefore Theorem \ref{thm3.1}
cannot be applied to \eqref{eqn4.10}. Nevertheless, we
will see in the following that the uniqueness theorem is still valid
 for \eqref{eqn4.10} provided that $\epsilon$ is small enough.

Since for $\rho=0$ the linearized problem has positive first eigenvalue, then, of course,
there is a small $\rho_0 > 0$, which do not depend on $\epsilon$, such that there is only one solution to
\eqref{eqn4.10} for $\rho \leq \rho_0$. For $\rho >\rho_0$ we set

\begin{eqnarray*}
u^*_\epsilon(x)=
\left\{%
\begin{array}{ll}
    u^\epsilon(x)-8\pi\epsilon [G(x,\,p)+\displaystyle\frac{1}{2\pi}\log|x-p|], & \hbox{if}\,x\in\O \\
    -4\epsilon\log|x-p|, & \hbox{if}\,x\in\O^{*}\setminus \O, \\
\end{array}%
\right.
\end{eqnarray*}
where $u^\epsilon(x)$ is a solution for \eqref{eqn4.10}. Then $u^*_\epsilon(x)\in
C(\Omega^*)$ and
\begin{eqnarray}\label{eqn4.11}
\Delta
u^*_\epsilon(x)+\rho\frac{h^{1-\epsilon}(x)e^{u^*_\epsilon(x)}}{\int_\O
h^{1-\epsilon}e^{u^*_\epsilon(x)}dx}\geq0\,\,\hbox{ in
}\,\O^*,
\end{eqnarray}
in the distribution sense, provided that the following holds:
\begin{eqnarray}\label{eqn4.12}
\frac{\partial}{\partial \nu}[u^\epsilon(x)-8\pi\epsilon
G(x,\,q)]\leq0 \,\hbox{ for }\,x\in \O^*\cap\partial\O.
\end{eqnarray}

We prove \eqref{eqn4.12} by contradiction. Suppose that there exists a sequence of
solutions $u^{\rho_k}_{\epsilon_k}$ for
\eqref{eqn4.10}$\equiv$\eqref{eqn4.10}$_{\epsilon_k}$ with $\rho_k\geq\rho_0$ such
that
\begin{eqnarray}\label{eqn4.13}
\frac{\partial
u^{\rho_k}_{\varepsilon_k}(x_k)}{\partial\nu}\geq-C\epsilon_k
\,\hbox{ for some}\, x_k\in\partial\Omega,
\end{eqnarray}
where $-C=\inf_{x\in\partial\Omega}\frac{\partial G(x,q)}{\partial\nu}$.

If $u^{\rho_k}_{\epsilon_k}$ is uniformly bounded in $\overline{\Omega}$, then there is a subsequence of
$u^{\rho_k}_{\epsilon_k}$, which converges to a function $u$ which satisfies
\begin{eqnarray*}
\left\{%
\begin{array}{ll}
    \Delta u + \rho^*_0\displaystyle\frac{h_{\epsilon_0^*}(x)e^{u(x)}}
{\int_{\O}h_{\epsilon^*_0}(x)e^{u(x)}dx}=0 \quad \hbox{in}\quad\O, \\
    u=0, \, \hbox{on}\,\partial\O,  \,\hbox{ and }\,\frac{\partial u}{\partial\nu}(x_0)\geq 0 \quad
\hbox{for some}\quad x_0\in\partial\O, \\
\end{array}%
\right.
\end{eqnarray*}
where
$\epsilon^*_0=\displaystyle\lim_{k\rightarrow+\infty}\epsilon_k$,
$\rho^*_0=\displaystyle\lim_{k\rightarrow+\infty}\rho_k\geq\rho_0>0$
and $x_0 = \displaystyle\lim_{k\rightarrow+\infty} x_k$. Since
$u>0$ on $\partial\Omega$, then the Hopf boundary Lemma says that $\frac{\partial u}{\partial\nu}(x)<0$ for all
$x\in\partial\Omega$, which is a contradiction to $\frac{\partial u}{\partial\nu}(x_0) \geq 0$.

On the other side, if there exists a blowing up subsequence (which we denote by $u_k$) of
$u^{\rho_k}_{\epsilon_k}$, then $\rho_k\rightarrow
8\pi,\epsilon_k\rightarrow\epsilon^*_0$, and
$\rho_k\frac{h_{\epsilon_k}(x)e^{u_k}}{\int_\Omega
h_{\epsilon_k}(x)e^{u_k}dx}\rightharpoonup 8\pi\delta_{x_1}$ for
some $x_1\in\Omega$. Furthermore, $u_k\rightarrow G(x,\,x_1)\,
\hbox{ in }\, C^2(\overline{\Omega}\setminus\{x_1\})$. At this point \eqref{eqn4.13} implies
$\frac{\partial G(x_0,x_1)}{\partial\nu}=\displaystyle\lim_{k\rightarrow+\infty}
\frac{\partial u_k}{\partial\nu}(x_k)\geq0$, where
$x_0=\displaystyle\lim_{k\rightarrow+\infty}x_k$, which is once more
a contradiction to the Hopf boundary Lemma. Hence \eqref{eqn4.12}
holds for any $\epsilon$ small enough.

Since $u_\epsilon^*$ satisfies the differential inequality
\eqref{eqn4.11}, we can follow the proof of Theorem \ref{Thm2.1}
to show that the Bol's inequality holds for $u^\epsilon$, i.e.
for any $\omega\Subset\O$, we have
$$
2\ell^2_\epsilon(\partial\omega)\geq
m_\epsilon(\omega)(8\pi-m_\epsilon(\omega)),
$$
\[
 \hbox{where}\ \
\ell_\epsilon(\partial\omega)=\int_{\partial\omega}e^{v_\epsilon(x)/2}ds,\,\,m_\epsilon(\partial
\omega)=\int_\omega e^{v_\epsilon(x)}ds,
\]
and
\[
v_\epsilon(x)=u^\epsilon(x)+\log{
h_\epsilon(x)}+\log\rho-\log\int_{\Omega} h_\epsilon(x)e^{u^\epsilon(x)}dx.
\]

By using the Bol's inequality, we can follow the proof of Theorem
\ref{thm3.1} to show that equation \eqref{eqn4.10} admits at most one solution for
$\rho\leq8\pi$ as well for any $\epsilon$ small enough.

Since $D_\epsilon(q)<0$, then we can apply the result obtained in \textbf{Case 1}. Therefore
the solution $u^\rho_\epsilon(x)$ of \eqref{eqn4.10} blows up at $q$ as $\rho\nearrow 8\pi$.
However, the same argument adopted in \textbf{Case 1} shows that $u^\rho(x)$, the solution of
\eqref{eqn2.1}, blows up at $q$ as well, which is the desired result.
\end{proof}

\begin{remark}\label{rem.qmax2}
The proof of Theorem \ref{thm4.1} shows in particular
that if there exists a maximum point $q$ of $\log h(x)+4\pi\gamma(x)$ with $D_h(q)>0$, then we have
$D_h(p)>0$ for any other maximum point $p$. See Lemma \ref{lem4.1} and the few lines above it.
\end{remark}

The following results will provide us with a proof of Corollaries \ref{cor1.3-add} and
\ref{cor1.3} in the more general situation where \eqref{eqn2.1} is concerned. The situation where
$\om$ is simply-connected has been already discussed in \cite{key02} and we will not pursue it here
any further. The nondegeneracy of the maximum point as stated in Corollary \ref{cor1.3-add} requires
a more subtle analysis which is the content of Theorem \ref{thm4.2} below.

\begin{corollary}\label{cor4.1} Let $\Omega\subset \R^2$ be an open and bounded domain of class $C^{1}$.
If $q$ is a critical point of $\log h(x)+4\pi\gamma(x)$ with
$D_h(q)\leq 0$, then $q$ is a maximum point. Furthermore, $q$ is the
unique maximum point.
\end{corollary}

\begin{proof}
By Lemma \ref{lem4.1} and Proposition \ref{prop4.1} we see that $q$ is the blow up point of a sequence of
blowing up solutions. Therefore $q$ is a maximum point, see Remark \ref{rem:qmax}. To prove
the uniqueness of any such maximum point, observe that Remark \ref{rem:qmax}
says that indeed $D_h(p)\leq 0$ for any other maximum point. Now suppose
that there exists another maximum point $q'\neq q$. Then Lemma \ref{lem4.1}
yields a sequence of blowing up solutions $u_{\rho_n'}$ whose blow up point
should be $q'$. Of course, this is a contradiction to Proposition \ref{prop4.1}.
\end{proof}

\begin{corollary}\label{cor4.2} Let $\Omega\subset \R^2$ be an open and bounded domain of class $C^{1}$.
If $\,\log h(x)+4\pi\gamma(x)$ admits more than one maximum point,
then equation \eqref{eqn2.1} has a solution for $\rho=8\pi$.
\end{corollary}

\begin{proof}
Let $q_1\neq q_2$ be maximum points of $\log h(x)+4\pi\gamma(x)$.
We deduce from Corollary \ref{cor4.1} that $D_h(q_1)>0$ and $D_h(q_2)>0$. Hence
Theorem \ref{thm4.1} yields the existence of a solution for \eqref{eqn2.1}.
\end{proof}

Now we are in the position to prove:

\begin{theorem}\label{thm4.2}
Let $q$ be a critical point of $\,\log h(x)+4\pi\gamma(x)$ with
$D_h(q)\leq 0$. Then $q$ is a nondegenerate critical point.
\end{theorem}

\begin{proof}
Since $D_h(q)\leq 0$ we deduce from Corollary \ref{cor4.1} that $q$ is the
unique maximum point. We argue by contradiction and suppose that $q$ is degenerate.
Without loss of generality, we may assume that $q=0$ and
\beq\label{D2.1}
\left.\frac{\partial^2}{\partial x_1^2}(\log h+4\pi\gamma)\right|_{x=0} =
\left.\frac{\partial^2}{\partial x_1\partial x_2}(\log h+4\pi\gamma)\right|_{x=0}= 0,
\eeq
\beq\label{D2.2}
\left.\frac{\partial^2}{\partial x_2^2}(\log h+4\pi\gamma)\right|_{x=0} = a \leq 0.
\eeq
The proof will be divided in two cases.

\textbf{Case 1:} $D_h(q)<0$.\\
We set
\beq\label{D2.3}
\log h_\varepsilon(x) = \log h(x)+\varepsilon(x_1^2-x_2^2),
\eeq
and let $D_\varepsilon(q)$ be defined by \eqref{eqn4.5} where $h$ has
just been replaced by $h_\varepsilon$.
Then $D_\varepsilon(q)<0$ if $\varepsilon$ is sufficiently small.
Since $x_1^2-x_2^2$ is harmonic, then Corollary \ref{cor4.1} can be applied and we
conclude that $q=0$ should be a maximum point. This is impossible as one readily verifies
 by using \eqref{D2.1} and \eqref{D2.2} together with \eqref{D2.3}. Hence the desired conclusion in Case 1
is established.

\textbf{Case 2:} $D_h(q)=0$.\\
We set
$$
h_t(x)=h(x)e^{-t x_1^4},\ \ t>0.
$$
Thus $q=0$ is a critical point of $h_t(x)$ and

$$
D_{h,t}(0)=\lim_{\varepsilon\rightarrow 0}\int_{\Omega\setminus
B(0;\varepsilon)}
\frac{h(x)}{h(0)}\dfrac{e^{-t{x_1}^4}e^{8\pi(\tilde{G}(x,0)-\gamma(0))}-1}{|x|^4}
- \int_{\Omega^c}\frac{dx}{|x|^4}.
$$

\noindent Clearly, $D_{h,t}(0)<D_h(0)$ for $t>0$. Next we consider the mean field equation:
\begin{equation*}
\hspace{-3.2cm} (4.17)_t\ \ \ \
    \left\{
\begin{array}{ll}
    \Delta u^t+\rho\dfrac{h(x)e^{-t{x_1}^4}{e^u}^t}{\int_\Omega h(x)e^{-t{x_1}^4}{e^u}^t dx}=0 & \hbox{in} \;\Omega, \\
    u^t=0  & \hbox{on }\;\partial\Omega. \\
\end{array}
\right.
\end{equation*}

Although $h_t$ is not subharmonic, we can argue as in Case 2 of the proof of Lemma \ref{lem4.1} to
show that $(4.17)_t$ admits at most one solution
for $\rho\in[0,8\pi]$ and for any $t$ small enough. For a fixed small
$t>0$, $D_{h,t}(0)<0$ and then the conclusion obtained in Case 1 above says that
$q=0$ is a nondegenerate critical point of $\log h(x)+4\pi\gamma(x)-t{x_1}^4$. Hence, it should be
a nondegenerate critical point of $\log h(x)+4\pi\gamma(x)$ as well, which is the desired contradiction in
Case 2.
\end{proof}

Finally we have the following generalized version of Corollary \ref{cor25.06}
\begin{corollary}\label{cor25.06-gen} Let $\Omega\subset \R^2$ be an open, bounded and
multiply-connected domain of class $C^{1}$. Then
$$
\frac{1}{8\pi}\inf_{u\in H^1_0(\O)} I_\rho(u) \leq-1-\log(\pi)-
\sup_{x\in\overline{\Omega}}\left(\log{h(x)}+4\pi\gamma(x)\right),
$$
and \eqref{eqn2.1} admits a solution at $\rho=8\pi$ if and only if the strict inequality holds.
\end{corollary}
\begin{proof}
Theorem \ref{thm4.1} says that a solution at $\rho=8\pi$ exists if and only if $D_h(q)>0$
for a maximum point of $\log{h(x)}+4\pi\gamma(x)$. Hence Lemma \ref{lem25.06} shows immediately that if a
solution exists, then the inequality is strict.\\
On the other side, if we assume by contradiction that the inequality is strict but no solution exists at
$\rho=8\pi$, then we get a contradiction to \eqref{eqn4.2}.
\end{proof}

\section{Equivalence of Statistical ensembles.}\label{sec5}
Our main concern in this section is the applicability of Theorem \ref{Thm1.2}
to some long standing open problems in the statistical mechanics analysis of two dimensional turbulence
\cite{kof03}.
This is why in some statements we will assume the domain's regularity taken up in Theorem \ref{Thm1.2},
see Theorems \ref{Thm5.3} and \ref{Equiv26.06} below. Let
$$
\s(t)=\graf{-t\log{t},\quad &t>0\\ 0,\quad &t=0, \ }
$$
and $\om\subset\rdue$ be a bounded domain. We define
$$
\mathcal{P}_{\scp \om}=\left\{\rho\in L^{1}(\om)\,|\,\rho\geq 0\;\mbox{a.e. in}\;\om,\;\ino\rho =1,
\ino(-\s(\rho))<+\infty \right\},
$$
and let $G(x,y)$ be the Green's function on $\om$ as defined in the introduction.
For any $\rho\in \mathcal{P}_{\scp \om}$ let us set
$$
\mathcal{S}(\rho)=\ino\s(\rho),\quad
\mathcal{E}(\rho)=\frac12 \ino \rho G[\rho],
$$
where
$$
G[\rho](x)=\ino G(x,y)\rho(y)\,dy,
$$
and
$$
\mathcal{F}_\be(\rho)=-\frac{1}{\be}\mathcal{S}(\rho)+\mathcal{E}(\rho).
$$

For any  $E\in\erre$ we consider the Microcanonical Variational Principle (MVP for short)
\beq\label{mvp}
S(E)=\sup\{\mathcal{S}(\rho),\; \rho \in \mathcal{P}_{\scp \om}(E) \},\quad
\mathcal{P}_{\scp \om}(E)=\{\rho\in\mathcal{P}_{\scp \om}\,|\,\mathcal{E}(\rho)=E\},
\eeq

\noindent
while for any  $-8\pi \leq \be<0$ we consider the Canonical Variational Principle (CVP for short)
\beq\label{cvp}
f(\be;\om)=\sup\{\mathcal{F}_\be(\rho),\; \rho \in \mathcal{P}_{\scp \om} \}.
\eeq

For each $0<\lm\leq 8\pi$ let us set
$$
g_\lm(\om):=\sup\limits_{u\in H^{1}_0(\om)} \mathcal{J}_\lm(u),\quad \mathcal{J}_\lm(u)=-\frac{1}{2}\ino |\nabla u|^2+\frac{1}{\lm}
\log\left(\ino e^{\lm u}\right).
$$

The following results are well known. Although not essentials for the incoming discussion,
they are quite relevant to understand how $f(\be;\om)$ (the physical free energy) is related
with $g_\lm(\om)$ (which is essentially our functional $J_\rho$). Because of this
basic role, we will provide a sketchy proof for the sake of completeness.
However we will not discuss the case where $\be>0$ (that is $\lm<0$) which is much easier, see \cite{kof03}.

\begin{theorem}\label{thm5.1} Let $\om\subset \rdue$ be an open and bounded domain satisfying a uniform cone property
\cite{Adams}. For any $-8\pi<\be<0$ and for
any $0< \lm<8\pi$ the supremums $f(\be;\om)$ and $g_\lm(\om)$ are
attained and in particular, setting $\lm=-\be$,
$$
f(\be;\om)=g_\lm(\om).
$$
\end{theorem}
\proof
By using Theorem 2.1 in \cite{kof02} we see that for any $-8\pi<\be<0$ the supremum
$f(\be;\om)$ is attained by a density
$\rho_{\scp \be}\in \mathcal{P}$ which solves the MFE (Mean Field Equation)
$$
\rho_{\scp \be}(x)=\frac{e^{-\be G[{\dsp \rho}_{\scp \be}](x)}}{\ino e^{-\be G[{\dsp \rho}_{\scp \be}]}},
\qquad\qquad Q(\be,\om).
$$

Actually the argument in \cite{kof02} relies on the evaluation of the thermodynamic limit for a renormalized
free energy functional. It turns out that another proof of this fact which uses variational type arguments based on the
logarithmic Hardy-Sobolev inequality \cite{lieb} can be found in Lemmas 2.1 and 2.2 in \cite{CSW} under stronger
regularity assumptions on $\om$.\\
The Moser-Trudinger inequality \cite{moser} and the direct method in the calculus of variations show that
for any $0< \lm<8\pi$ the supremum $g_\lm(\om)$ is attained by a function $v=v_{\scp \lm}\in H^1_0(\om)$
which solves the MFE
$$
\graf{
-\Delta v =\dfrac{e^{\lm v}}{\ino e^{\lm v}} & \mbox{in}\quad \om\\
v =0 & \mbox{on}\quad \pa\om
}\qquad\qquad P(\lm,\om).
$$

Let  $\rho_{\scp \be}$ be a maximizer for $\be>-8\pi$, put $\be=-\lm$ and $v_\lm=G[\rho_{\scp -\lm}]$.
Then we see that $v_\lm$ is in $H^{1}_0(\om)$ and solves $P(\lm,\om)$. Hence, by a straightforward evaluation,
we obtain
$$
f(-\lm;\om)=\mathcal{F}_{-\lm}(\rho_{\scp -\lm})=\mathcal{J}_{\lm}(v_\lm).
$$
Clearly, the same equality, to be read in the opposite direction, shows that if $v=v_\lm$ is a maximizer
for $\lm<8\pi$ and we define $\rho_{\scp -\lm}=\frac{e^{\lm {\dsp v}_\lm}}{\ino e^{\lm {\dsp v}_\lm}}$, then
it clearly solves $Q(-\lm,\om)$ and in particular
$$
g_\lm(\om)=\mathcal{F}_{-\lm}(\rho_{\scp -\lm}).
$$
We easily deduce at this point that $g_\lm(\om)$ and $f(-\lm;\om)$ must coincide.\fineproof

\bigskip

\begin{remark}\label{MFE-equiv}
The proof above shows in particular that $\rho_{\scp \be}$ solves $Q(\be,\om)$ if and only if, setting $\lm=-\be$,
then $v_\lm=G[\rho_{\scp-\lm}]$ belongs to $H^1_0(\om)$ and weakly solves $P(\lm,\om)$.
Clearly $P(\lm,\om)$ is equivalent to problem \eqref{eqn1.1} as far as $\lm\neq 0$.
\end{remark}

\bigskip

\begin{remark} Since $\mathcal{S}$ is concave, if $\rho\in\mathcal{P}_{\scp \om}$, by the Jensen's inequality we have
$$
0=\s\left(\ino \rho\right)\geq
\ino\s(\rho)=\mathcal{S}(\rho),
$$
that is, $\mathcal{S}(\rho)\leq 0$, $\forall \rho\in\mathcal{P}_{\scp \om}$.
\end{remark}

\begin{theorem}\label{Thm5.2}  Let $\om\subset \rdue$ be an open and bounded domain. For any $E>0$,
$S(E)<+\infty$ and there exists $\rho^{(E)}\in \mathcal{P}_{\scp \om}(E)$ such that
$S(E)=\mathcal{S}(\rho^{(E)})$. In particular $S(E)$ is continuous, and letting $E_0=\mathcal{E}\left(\frac{1}{|\om|}\right)$,
then $S(E_0)=0$ and $S$ is strictly increasing for $E<E_0$ and strictly decreasing for $E>E_0$.
Moreover for each $E$ there exists $\be\in\erre$ such that $\rho^{(E)}$ solves $Q(\be,\om)$.
\end{theorem}
\proof This is the content of Propositions 2.1, 2.2, 2.3 and 2.4 in \cite{kof03}.\fineproof

\bigskip

Theorem \ref{Thm5.2} shows that the MVP always has a solution (that is, the supremum in \eqref{mvp}
is always attained) which consequently describes the (mean field) thermodynamic of the system.
This is no longer true for the CVP which surely won't
have a solution if $\be<-8\pi$ (that is, the supremum in \eqref{cvp} is not attained).
The characterization of those cases where these two formulations
yield the same thermodynamics (\underline{equivalence} of microcanonical and canonical statistical ensambles)
is one of the main aim in the statistical mechanics description of the system. The following
results are concerned with the solution of this problem.\\

Let us recall that $P(\lm,\om)$ has been introduced during the proof Theorem \ref{thm5.1} to denote
one of the equivalent formulations of the mean field equation, see Remark \ref{MFE-equiv}.\\
In \cite{kof03}, a bounded domain is said to be:\\
\noindent
(-) of {\bf first kind}, if solutions of $P(\lm,\om)$ blow up as $\lm\nearrow 8\pi^{-}$;\\
(-) of {\bf second kind}, otherwise.\\
Among many other things, it was shown in \cite{kof03} that if the inequality is strict in Corollary
\ref{cor25.06} then $\Omega$ is of second kind. On this basis some examples of simply-connected domains
of second kind were exhibited there, but a full characterization was still missing. This problem was
solved in \cite{key02}, where Chang, Chen and the second author characterized domains of the first/second
kind in case $\om$ is simply-connected. More recently this result has been extended in \cite{BL2} to cover
the case where Dirac data are included in \eqref{eqn2.1}.
We complete those results here with a full characterization of domains of
first/second kind. In particular the following Theorem
has to be complemented with Theorem \ref{Thm1.1} and Corollary \ref{cor25.06}
(and Theorems 1.1 and 1.5 in \cite{key02}) which provide other necessary and
sufficient conditions for a fixed domain to be of first or second kind.
For example, it has been already observed in the introduction that
the annulus $B(0,1)\setminus B(x_0,\epsilon)$ with $\varepsilon<1-|x_0|$ is of second kind if
$x_0=0$ while if $x_0\neq 0$ and $\varepsilon$ is small enough then it is of first kind.

\begin{theorem}\label{Thm5.3}  Let $\om$ be an open, bounded domain of class $C^1$.  The following facts are equivalent:\\
(-) $\om$ is of first kind;\\
(-) $g_{8\pi}(\om)$ is not attained;\\
(-) $P(8\pi,\om)$ has no solution;\\
(-) The unique branch of maximizers for $g_{\lm}(\om)$, $\lm<8\pi$ blows up as $\lm\nearrow 8\pi$.\\
\noindent Moreover, the following facts are equivalent:\\
(-) $\om$ is of second kind;\\
(-) $g_{8\pi}(\om)$ is attained;\\
(-) $P(8\pi,\om)$ admits a solution $u_{8\pi}$;\\
(-) The unique branch of maximizers $u_\lm$ for $g_{\lm}(\om)$, $\lm<8\pi$ converges uniformly to $u_{8\pi}$
as $\lm\nearrow 8\pi$.
\end{theorem}
\proof By using Theorem \ref{Thm1.2} and the implicit function Theorem the proof can be worked out as in
Proposition 6.1 in \cite{key02}.\fineproof

\begin{remark}
We remark that if $\om$ is simply-connected then Theorem \ref{Thm5.3} holds even if $\om$ has a finite number of
conical-type singular points, see \cite{key02} for further details and a discussion
of the first/second kind issue for some natural domains such as rectangles and polygons.
\end{remark}

By using the uniqueness of solutions for $P(\lm,\om)$ with $\lm<8\pi$ obtained in \cite{key20}, and
 under certain further assumptions on the topology of $\om$, in \cite{kof03}
the authors were able to establish the \un{equivalence} of statistical ensembles, namely the above mentioned
equivalence of the variational principles \eqref{mvp} and \eqref{cvp}. Indeed, since
the uniqueness in \cite{key20} was obtained just for simply-connected (smooth and bounded) domains, they restrict
their attention to this class. By using Theorem \ref{Thm1.2} we are able to extend
those results to the general case of bounded domains of class $C^1$. As a matter of fact, the proof
adopted in \cite{kof03} works fine as well, the unique modification being just that of using
Theorem \ref{Thm1.2} instead of the Suzuki's \cite{key20} uniqueness result.
This is why we will not repeat those proofs in full details here.
We remark that in \cite{key02} and \cite{BL1} the full uniqueness theory
presented here was developed under much weaker smoothness assumptions on $\om$. Let
$$
E=E(\be)=\mathcal{E}(\rho_{\scp \be}),
$$
be the energy of the (unique, see Remark \ref{MFE-equiv} and Theorem \ref{Thm1.2})
solution of $P(\lm,\om)$ where $\lm=-\be$ for $0<\lm<8\pi$. Hence $E:(-8\pi,0)\mapsto \R^+$ is well defined.
As in \cite{kof03},
if $\om$ is of first kind we set $E_c=E(-8\pi)=+\infty$, while if it is of second kind we set
$E_c=E(-8\pi)<+\infty$. Once more, as already mentioned above, we will not discuss the situation where
$\be\geq 0$ which is easier, see \cite{kof03}.\\
We finally have the generalization of Proposition 3.3 in \cite{kof03} to
the case where $\om$ is an open, bounded and multiply connected domain of class $C^1$.

\begin{theorem}\label{Equiv26.06}  Let $\om$ be an open and bounded domain of class $C^1$.
We assume that either $\om$ is of first kind and
$E\in (0,+\infty)$ or $\om$ is of second kind and $E\in(0,E_c)$. Then we have:\\
(i) $F(\be)=-\be f(\be;\om)$ is defined for $\be\geq -8\pi$, strictly convex and decreasing;\\
(ii) $F$ is differentiable for $\be>-8\pi$ and $E(\be)=-F^{'}(\be)=\frac12 \ino \rho_{\scp \be}G[\rho_{\scp \be}]$, where
$\rho_{\scp \be}$ solves $Q(\be,\om)$. In particular $E(\be)$ is a continuous and strictly monotone decreasing bijection;\\
(iii) $S(E)=\inf\limits_{\be}\{F(\be)+\be E(\be)\}$ and hence is a smooth and concave function of $E$;\\
(iv) If $\rho^{(E)}$ is a maximizer for \eqref{mvp} then $\rho_{\scp \be}=\rho^{(E(\be))}$. In particular
$\rho^{(E)}$ solves $P(-\be;\om)$ (\un{equivalence of statistical ensembles for $E<E_c$}) and the solution is unique.
\end{theorem}
\proof
Of course, we restrict our attention to the case where $\om$ is multiply-connected, the other case
being already included in \cite{kof03} Proposition 3.3.\\
(i) The proof can be worked out as in Proposition 7.3 in \cite{kof02}.\\
(ii) We argue as in \cite{kof03} making use of Theorem \ref{Thm1.2} above.
Let $-8\pi<\be_i<0$, $i=1,2$ and $\rho_i$, $i=1,2$ be the corresponding maximizers
of \eqref{cvp}. Clearly
$$
E(\be_i)=\mathcal{E}(\rho_i),\quad i=1,2.
$$
Hence
$$
F(\be_2)\geq \mathcal{S}(\rho_1)-\be_2\mathcal{E}(\rho_1)\equiv F(\be_1)-(\be_2-\be_1)E(\be_1),
$$
$$
F(\be_1)\geq \mathcal{S}(\rho_2)-\be_1\mathcal{E}(\rho_2)\equiv F(\be_2)-(\be_1-\be_2)E(\be_2).
$$
\noindent
The last two inequalities imply
$$
-E(\be_2)\leq \dfrac{F(\be_2)-F(\be_1)}{\be_2-\be_1}\leq -E(\be_1),\quad \mbox{if}\quad \be_1>\be_2,
$$
and
$$
-E(\be_1)\leq \dfrac{F(\be_2)-F(\be_1)}{\be_2-\be_1}\leq -E(\be_2),\quad \mbox{if}\quad \be_2>\be_1.
$$
\noindent
Let $v_i:=G[\rho_i]$, $i=1,2$ be the corresponding solutions of $P(\lm_i,\om)$, $\lm_i=-\be_i$, $i=1,2$.
Then, in view of Remark \ref{MFE-equiv} and Theorem \ref{Thm1.2}
we see that, as $\be_1\to\be_2$, then $v_1:=G[\rho_1]\to v_2:=G[\rho_2]$ in $H^1_0(\om)$. In particular
it is not difficult to verify that
$$
E(\be_1)=\frac12 \ino \rho_1 v_1\to \frac12 \ino \rho_2 v_2=E(\be_2),\quad \mbox{as}\quad \be_1\to\be_2.
$$
Therefore we conclude that
$$
-F^{'}(\be)=E(\be)=\frac12 \ino \rho_{\scp \be} G[\rho_{\scp \be}],\quad \forall\, \be\in(-8\pi,0).
$$
\noindent
The continuity of $E(\be)$ follows once more from Theorem \ref{Thm1.2}.
Finally $E=-F^{'}$ is strictly monotone and decreasing (hence a bijection)
since $F$ is strictly convex.\\
(iii)-(iv) In view of (ii), the proof provided in Proposition 3.3(iii)-(iv) of \cite{kof03} works exactly as it stands.
\fineproof

\end{document}